\documentclass{amsart}
\usepackage[margin=3.5cm]{geometry}
\usepackage{mathtools}
\usepackage{amssymb}
\usepackage{enumerate}             
\newtheorem{theorem}{Theorem}[section]
\newtheorem{intro}{Theorem}
\newenvironment{introthm}[1]
  {\renewcommand\theintro{#1}\intro}
  {\endintro}
\newtheorem{prop}[theorem]{Proposition}
\newtheorem{cor}[theorem]{Corollary}
\newtheorem{lem}[theorem]{Lemma}

\theoremstyle{definition}
\newtheorem*{rem}{Remark}
\newtheorem*{notation}{Notation}
\numberwithin{equation}{section}

\newcommand{\thmref}[1]{Theorem~\ref{#1}}
\newcommand{\lemref}[1]{Lemma~\ref{#1}}
\newcommand{\propref}[1]{Proposition~\ref{#1}}

\newcommand{\bt}{\mathbf t}                   
\newcommand{\bu}{\mathbf u}                   
\newcommand{\dN}{\mathbb N}                   
\newcommand{\dZ}{\mathbb Z}                   
\newcommand{\dR}{\mathbb R}                   
\newcommand{\dC}{\mathbb C}                   
\newcommand{\ra}{\rightarrow}                 
\newcommand{\tO}{\mathtt 0}                   
\newcommand{\tL}{\mathtt 1}                   

\let\abs=\envert
\let\norm=\enVert

\newcommand{\ang}[1]{\left\langle#1\right\rangle}       
\newcommand{\floor}[1]{\left\lfloor#1\right\rfloor}     
\newcommand{\ceil}[1]{\left\lceil#1\right\rceil}        
\newcommand{\p}[1]{\left( #1 \right)}                   
\newcommand{\e}{\operatorname{e}}                       
\DeclareMathOperator{\logp}{\log^+\!}                   
\allowdisplaybreaks
\begin{document}
\title[Normality of the Thue--Morse sequence]%
{Normality of the Thue--Morse sequence along Piatetski-Shapiro sequences,~II}

\author[C. M\"ullner]{Clemens M\"ullner}
\thanks{
The first author was supported by the Austrian Science Fund (FWF), Project F5502-N26,
which is a part of the Special Research Program
``Quasi Monte Carlo Methods: Theory and Applications'',
and Project F5002-N15, which is a part of the Special Research Program
``Algorithmic and Enumerative Combinatorics''.}

\address{Institute of Discrete Mathematics and Geometry,\\
Vienna University of Technology,
Vienna, Austria}
\author[L. Spiegelhofer]{Lukas Spiegelhofer}
\thanks{
The second author was supported by the Austrian Science Fund (FWF), Project F5502-N26,
which is a part of the Special Research Program
``Quasi Monte Carlo Methods: Theory and Applications'', and Project I1751,
called MuDeRa (Multiplicativity, Determinism, and Randomness).}

\address{Institute of Discrete Mathematics and Geometry,\\
Vienna University of Technology,
Vienna, Austria}

\begin{abstract}
We prove that the Thue--Morse sequence $\mathbf t$ along subsequences indexed by
$\lfloor n^c\rfloor$ is normal, where $1<c<3/2$.
That is, for $c$ in this range and for each $\omega\in\{0,1\}^L$,
where $L\geq 1$,
the set of occurrences of $\omega$ as a factor (contiguous finite subsequence)
of the sequence $n\mapsto \mathbf t_{\lfloor n^c\rfloor}$
has asymptotic density $2^{-L}$.
This is an improvement over a recent result by the second author,
which handles the case $1<c<4/3$.

In particular, this result shows that for $1<c<3/2$ the sequence
$n\mapsto \mathbf t_{\lfloor n^c\rfloor}$
attains both of its values with asymptotic density $1/2$,
which improves on the bound $c<1.4$ obtained by Mauduit and Rivat
(who obtained this bound in the more general setting of
$q$-\emph{multiplicative functions}, however)
and on the bound $c\leq 1.42$ obtained by the second author.

In the course of proving the main theorem, we show that $2/3$
is an \emph{admissible level of distribution} for the Thue--Morse sequence,
that is, it satisfies a Bombieri--Vinogradov type theorem for each exponent
$\eta<2/3$.
This improves on a result by Fouvry and Mauduit,
who obtained the exponent $0.5924$.
Moreover, the underlying theorem implies that every finite word
$\omega\in\{0,1\}^L$ is contained as an arithmetic subsequence of $\mathbf t$.
\end{abstract}
\maketitle
\section{Introduction}
The Thue--Morse sequence $\bt$ is a well-known infinite sequence on the two symbols $\tO$ and $\tL$,
which can be defined as follows.
Starting with the $1$-element sequence $\bt^{(0)}=(\tO)$
and constructing $\bt^{(n+1)}$ by concatenating $\bt^{(n)}$
and its Boolean complement $\neg\bt^{(n)}$,
the infinite sequence
\[\bt=\tO\tL\tL\tO\tL\tO\tO\tL\tL\tO\tO\tL\tO\tL\tL\tO\ldots\]
is the pointwise limit of these finite sequences.
In other words, it is the fixed point, starting with $\tO$,
of the substitution $\tO\mapsto\tO\tL$,
$\tL\mapsto\tL\tO$.
The sequence $\bt$ can therefore be seen as a $2$-automatic sequence
(see the book~\cite{AS2003} by Allouche and Shallit),
indeed it is one of the simplest such sequences.
Another equivalent definition uses the \emph{binary sum-of-digits function} $s$,
which counts the number of $1$s in the binary expansion of a nonnegative integer:
we have $\bt_n=\tO$ if and only if $2\mid s(n)$.
Since $\bt$ is an automatic sequence,
its factor complexity is bounded above by a linear function,
that is, the number $P(k)$ of different factors of length $k$ of $\bt$
is bounded by $Ck$ for some constant $C$.
(For the Thue--Morse sequence, we have $\limsup P(k)/k=10/3$~\cite{B87}.)
More about the Thue--Morse sequence can be found in the article~\cite{AS99} by Allouche and Shallit,
which gathers occurrences of the Thue--Morse sequence in different fields of mathematics and offers a good bibliography,
and in the article~\cite{M2001} by Mauduit.

Although the sequence $\bt$ itself has low complexity,
the situation changes completely if we consider the subsequence indexed by the squares $0,1,4,9,16,\ldots$.
Moshe~\cite{M2007} proved that this subsequence has full factor complexity, that is,
every block $\{\varepsilon_0,\ldots,\varepsilon_{L-1}\}\in\{\tO,\tL\}^L$, for $L\geq 1$,
occurs as a factor of the sequence $n\mapsto \bt_{n^2}$.
Drmota, Mauduit and Rivat~\cite{DMR2015} proved the stronger statement that each of these blocks in fact occurs with density $2^{-L}$,
in other words, this subsequence is a \emph{normal sequence}.
This latter result was the motivation for our research.

The second author~\cite{S2015} recently proved
(using an estimate for discrete Fourier coefficients from~\cite{DMR2015})
an analogous result for subsequences indexed by so-called
\emph{Piatetski-Shapiro sequences},
which are sequences of the form $\floor{n^c}$.
\begin{introthm}{0}[Spiegelhofer 2015]\label{thm:S2015}
Assume that $1<c<4/3$. Then the sequence
$\left(\bt_{\floor{n^c}}\right)_{n\geq 0}$
is normal.
\end{introthm}
The study of Piatetski-Shapiro sequences, ``Polynomials of degree $c$'',
can be motivated by problems for polynomials of degree $2$.
For example, while it is unknown whether there are infinitely many primes of the form $n^2+1$,
the Piatetski-Shapiro Prime Number Theorem~\cite{P53} states that for $1<c<12/11$
the number of primes of the form $\floor{n^c}$
behaves asymptotically as one would expect by heuristic arguments.
The exponent $c$ has been improved several times,
the currently best known bound $c<2817/2426$ being due to Rivat and Sargos~\cite{RS2001}.

In a similar way the problem of studying the sum of digits of $\floor{n^c}$ can be motivated.
It was asked by Gelfond~\cite{G68} to investigate the distribution in residue classes of the sum of digits of values of polynomials $f$ such that $f(\dN)\subseteq \dN$.
This problem could not be solved even for polynomials of degree two,
and so Mauduit and Rivat~\cite{MR95,MR2005}
first proved a nontrivial result on sequences
$n\mapsto\varphi\bigl(\lfloor n^c\rfloor\bigr)$,
where $1\leq c<1.4$ and $\varphi$ is a $q$-\emph{multiplicative function} with values on the unit circle $S^1$.
We do not give a definition of this term, we only note that the Thue--Morse sequence in the form $n\mapsto (-1)^{s(n)}$ is a $2$-multiplicative function.
(The cited result~\cite{MR2005} was also transferred to automatic sequences by Deshouillers, Drmota and Morgenbesser~\cite{DDM2012}.
They also proved in that article that for $1<c<10/9$ blocks of length~$2$ in Piatetski-Shapiro subsequences of $\bt$ occur with frequency~$1/4$.)
Dartyge and Tenenbaum~\cite{DT2006} made some progress on the original question of Gelfond,
and finally Mauduit and Rivat~\cite{MR2009} could tackle the sum of digits of squares.
(This result was later generalized to compact groups by Drmota and Morgenbesser~\cite{DM2012}.)
However, there remained a gap to be closed---%
nothing was known on Piatetski-Shapiro subsequences of $q$-multiplicative functions for $c$ in the range 
$[1.4,2)$. 
In~2014, the second author~\cite{S2014} extended the bound for $c$ in the case of the Thue--Morse sequence,
obtaining the result that for $c\leq 1.42$ the sequence $n\mapsto\mathbf t_{\floor{n^c}}$ attains both of its values with asymptotic density $1/2$ (that is, this sequence is \emph{simply normal}).

In the present paper, we improve on the known results on normality and simple normality of Piatetski-Shapiro subsequences of $\bt$.
Our main theorem is the following.
\begin{introthm}{1}\label{thm:main}
Let $1<c<3/2$. Then the sequence $\bu=\bigl(\bt_{\floor{n^c}}\bigr)_{n\geq 0}$ is normal.
More precisely, for any $L\geq 1$ there exists an exponent $\eta>0$ and a constant $C$
such that
\[
\Bigl\lvert\abs{\bigl\{n<N:\bigl(\bu(n),\bu(n+1),\ldots,\bu(n+L-1)\bigr)=\omega\bigr\}}-N/2^L\Bigr\rvert
\leq CN^{1-\eta}
\]
for all $\omega=\bigl(\omega_0,\ldots,\omega_{L-1}\bigr)\in\{\tO,\tL\}^L$.
\end{introthm}
The essential innovation provided by this theorem lies in the new bound $3/2=1.5$, which replaces the bound $4/3=1.\dot{3}$ for (proper) normality as in \thmref{thm:S2015}.
For comparison, we note that $4/3$ is the bound that Mauduit and Rivat obtained in the first paper~\cite{MR95} on Piatetski-Shapiro subsequences of $q$-multiplicative functions, while $1.42$~\cite{S2014} is the most recent improvement on the exponent $c$, concerning simple normality of Piatetski-Shapiro subsequences of $\bt$.
The new bound $3/2$ established in our theorem therefore is a significant improvement---%
not only does it surpass the bound $1.42$ for simple normality (which is the case that $L=1$),
it also pushes the bound for proper normality beyond this value.
Another improvement on \thmref{thm:S2015} is the error term $CN^{1-\eta}$,
where both the exponent $\eta$ and the constant can be made completely explicit.

In the proof of \thmref{thm:main}, we reduce the problem of handling Piatetski-Shapiro sequences to the study of \emph{Beatty sequences} $\floor{n\alpha+\beta}$,
where $\alpha,\beta\in\dR$ and $n$ is contained in a small interval in $\dN$ of length $N$, a process that is basically Taylor approximation of degree~$1$~\cite{S2014}.
This yields sums of the form $\sum_{n\in I}\bt(\floor{n\alpha+\beta})$.
In order to deal with these sums,
we use methods by Mauduit and Rivat~\cite{MR2009,MR2010},
introducing the two-fold restricted sum-of-digits function.
Afterwards, we eliminate the Beatty sequences $\floor{n\alpha+\beta}$ from our expressions by exploiting the usually very small discrepancy (modulo~$1$) of $n\alpha$-sequences.
The resulting expression can be estimated nontrivially with the help of a new estimate (similar to~\cite[Proposition~1]{DMR2015}) for discrete Fourier coefficients related to the sum-of-digits function, which finishes the proof of \thmref{thm:main}.

Let us give a few more details on the methods used in the proof.
By Taylor approximation, we may approximate $\floor{n^c}$ on short intervals $I$ by Beatty sequences in such a way that $\floor{n^c}=\floor{n\alpha+\beta}$ for most $n\in I$.
This method is summarized in \propref{prp:PS_via_beatty_2} and leads to a statement reminiscent of the Bombieri--Vinogradov theorem in prime number theory:
in order to prove our theorem, we have to study the distribution of the Thue--Morse sequence on Beatty sequences $\floor{n\alpha+\beta}$,
where we take an average in $\alpha$ over dyadic intervals $[D,2D]$ (see \thmref{thm:2}).
Of course, greater values of $D$ correspond to greater exponents $c$ in the original problem.
Therefore we want to obtain a nontrivial distribution result for given length $N$ of the Beatty sequence, and $D$ as large as possible.
This Beatty sequence approach has been followed by the second author~\cite{S2014}.
In that article, trigonometric approximation of indicator functions was used in order to dispose of the Beatty sequences,
which led to the integral~\eqref{eqn:FM} in Section~\ref{sec:main}.
An estimate for this integral, taken from~\cite{FM96}, yielded the new bound $1.42$,
which surprisingly beat the formerly best bound $1.4$~\cite{MR2005} (concerning simple normality of subsequences of $\bt$).
However, we also have a lower bound on this integral, which sets a limit for this method.
In particular, $3/2$ can not be reached in this way.

In the second part of the proof of \thmref{thm:main} we therefore use a method different from trigonometric approximation in order to handle Beatty subsequences of $\bt$.
This method is based on the fact that Beatty sequences $\floor{n\alpha+\beta}$, for most $\alpha$,
are uniformly distributed in residue classes, their discrepancy being very small.
To put it simply, we will have to deal with sums
\begin{align}\label{eqn:beatty_sum}
\sum_{n\in I}f(\floor{n\alpha+\beta}),
\end{align}
where $f$ is $2^\gamma$-periodic and $I\subseteq\dZ$ is an interval slightly longer than $2^\gamma$, for instance $\abs{I}=2^{\gamma(1+\varepsilon)}$.
By a result on the mean discrepancy of $n\alpha$-sequences (\lemref{lem:mean_discrepancy}),
we may replace this sum, on average, by
\[    \frac{\abs{I}}{2^\gamma}\sum_{n<2^\gamma}f(n)+O\bigl(2^\gamma(\log \lvert I\rvert)^2\bigr).    \]
In order to apply this argument, we have to obtain sums of the form~\eqref{eqn:beatty_sum}.
To this end, we adapt methods by Mauduit and Rivat~\cite{MR2009,MR2010}, thereby
introducing the two-fold restricted sum-of-digits function $s_{\mu,\lambda}$ (which we define below).
The transition from the sum-of-digits function $s$ to the truncated version $s_\lambda$ is straightforward and can also be carried out for $\floor{n^c}$ (see~\cite{S2015}).
It is not so clear, however, how to get rid of the lowest $\mu$ digits of $\floor{n^c}$, that is, how to proceed from $s_\lambda(\floor{n^c})$ to $s_{\mu,\lambda}(\floor{n^c})$.
At this point, Beatty sequences $\floor{n\alpha+\beta}$ come into play:
using rational approximations to $\alpha$ and a generalization of van der Corput's inequality (\lemref{lem:vdc}),
it is possible to eliminate the lowest $\mu$ digits of $\floor{n\alpha+\beta}$,
so that we are left with only $\gamma:=\lambda-\mu$ binary digits of $\floor{n\alpha+\beta}$.
This further reduction of the number of digits to be taken into account ultimately allows us to achieve the improvement $3/2$ over the bound $4/3$ obtained in~\cite{S2015}.
As the last step of the proof, we use \propref{prp:uniform_fourier} from below.
This proposition is a new estimate for discrete Fourier coefficients, related to a result by Drmota, Mauduit and Rivat~\cite[Proposition~1]{DMR2015}, and allows us not only to deal with general block lengths $L\geq 1$, but also to derive an explicit error term of the form stated in \thmref{thm:main}.

The classical Bombieri--Vinogradov Theorem is a statement on the average distribution of prime numbers in arithmetic progressions $nd+j$, where the average is taken in the modulus $d$.
Let
\[  \psi(x;d,j) = \sum_{\substack{1\leq n\leq x\\n\equiv j\bmod d}}\Lambda(n),  \]
where $\Lambda$ is the von~Mangoldt function defined by $\Lambda(n)=\log p$ if $n=p^k$ for some prime $p$ and some $k\geq 1$ and $\Lambda(n)=0$ otherwise.
Then the Bombieri--Vinogradov Theorem~\cite[Theorem~4]{B65} states that for all positive real numbers $A>0$ there exist $B>0$ and a constant $C$ such that
\[
\sum_{1\leq d\leq D}
\max_{1\leq y\leq x}
\max_{\substack{j\in\dZ\\(j,d)=1}}
\abs{\psi(y;d,j)-\frac y{\phi(d)}}
\leq Cx(\log x)^{-A},
\]
where $D=x^{1/2}(\log x)^{-B}$.
While no improvement on the exponent $1/2$ is known, the Elliott--Halberstam conjecture~\cite{EH70} asks whether we can choose $D=x^{1-\varepsilon}$ for any $\varepsilon>0$.
In other words, it is conjectured that~$1$ is an \emph{admissible level of distribution for the primes}, whereas the largest known admissible level (as of~2015) equals $1/2$.
In the article~\cite{FM96}, Fouvry and Mauduit prove a Bombieri--Vinogradov type theorem for the sum-of-digits function $s$ in base $2$.
In particular, for the case of the Thue--Morse sequence they set
\[
A^{\pm}(x;d,j)=\Bigl\lvert\bigl\{n<x:(-1)^{s(n)}=\pm 1, n\equiv j\bmod d\bigr\}\Bigr\rvert
\]
and obtain
\begin{equation}\label{eqn:FM96}
\sum_{1\leq d\leq D}
\max_{1\leq y\leq x}
\max_{j\in\dZ}
\Bigl\lvert  A^{\pm}(y;d,j) - \frac y{2d}  \Bigr\rvert    \leq    Cx(\log 2x)^{-A}
\end{equation}
for all real $A$ and $D=x^{0.5924}$.
The exponent $0.5924$ can therefore be called \emph{admissible level of distribution for the Thue--Morse sequence}.
(We note that Fouvry and Mauduit obtain in fact an error term $x^{1-\eta}$ for some $\eta>0$, which follows from~\cite[Th\'eor\`eme~2]{FM96}.
We will use this improved estimate in the proof of Corollary~\ref{cor:1}.)
Using sieve theory, they apply this result to the study of the sum of digits of almost prime numbers, that is, integers that are the product of at most two prime factors.
Later and by different means, Mauduit and Rivat~\cite{MR2010} studied the sum of digits of prime numbers, which was not accessible by the Fouvry--Mauduit method.

As we indicated earlier, the backbone of our main result is a Bombieri--Vinogradov type theorem for $\bt$.
We establish $2/3$ as an admissible level of distribution for the Thue--Morse sequence, improving on the bound established by Fouvry and Mauduit.
A Beatty sequence version of this result, combined with linear approximation of $\floor{n^c}$, allows us to obtain the improvement $1.5$ over the bound $1.42$.
\begin{notation}
We use the common abbreviations $\e(x)=\exp(2\pi i x)$, $\{x\}=x-\floor{x}$, and $\norm{x}=\min_{n\in\dZ}\abs{x-n}$, where $x$ is a real number.
For a prime number $p$ let $\nu_p(n)$ be the exponent of $p$ in the prime factorization of $n$.
We define the \emph{truncated binary sum-of-digits function} \[    s_\lambda(n)=s(\tilde n),    \]
where $0\leq \tilde n<2^\lambda$ and $\tilde n\equiv n\bmod 2^\lambda$,
which only takes into account the digits of $n$ at positions smaller than $\lambda$,
and for $\mu\leq \lambda$ the \emph{two-fold restricted binary sum-of-digits function}
\[    s_{\mu,\lambda}(n)=s_\lambda(n)-s_\mu(n),    \]
which only depends on the digits at the positions $\mu,\ldots,\lambda-1$.
The functions $s_\lambda$ and $s_{\mu,\lambda}$ are periodic with period $2^\lambda$.
In estimates we use the convenient abbreviation
\[    \logp x=\max\left\{1,\log x\right\}.    \]
Summation variables are always assumed to be nonnegative.
In particular, we often omit conditions such as $0\leq n$ under summation signs.
In this article, the symbol $\dN$ denotes the set of nonnegative integers.
Moreover, constants implied by the symbols $\ll$ and $O$ depend at most on $L$,
that is, on the length of a factor $\omega=(\omega_0,\ldots,\omega_{L-1})$ of the sequences considered.
\end{notation}
\section{Results and overall structure of the proofs}\label{sec:main}
In the introduction we have already stated our main theorem (\thmref{thm:main}), concerning the normality of Piatetski-Shapiro subsequences of $\bt$ for exponents $c<3/2$.
In the current section we state auxiliary results used for proving this theorem:
approximation of $\floor{n^c}$ by Beatty sequences (\propref{prp:PS_via_beatty_2}),
a Beatty--Bombieri--Vinogradov theorem for $\bt$ (\thmref{thm:2} and its precursor, \propref{prp:2}),
and an estimate for discrete Fourier coefficients (\propref{prp:uniform_fourier}).
Moreover, we state results analogous to \thmref{thm:2} and \propref{prp:2},
concerning arithmetic progressions (\thmref{thm:1} and \propref{prp:1}), which follow from the same method of proof.

Let $\alpha,\beta,y$ and $z$ be nonnegative real numbers such that $\alpha\geq 1$, and
$\omega=(\omega_0,\ldots,\omega_{L-1})\in\{0,1\}^L$, where $L\geq 1$ is an integer.
We define
\begin{multline*}
A_\omega(y,z;\alpha,\beta)=\bigl\lvert\bigl\{y\leq m<z:\exists n\in\dZ\text{ such that }m=\floor{n\alpha+\beta}\text{ and}
\\
s(\floor{(n+\ell)\alpha+\beta})\equiv \omega_\ell\bmod 2\text{ for }0\leq \ell<L\bigr\}\bigr\rvert
.
\end{multline*}
Note that for $L=1$, $\alpha,\beta\in\dZ$ and $y=0$ this yields the sets
$A^\pm(x;d,j)$ that occur in~\cite{FM96}
(in that article, however, general moduli $q\geq 2$ are handled.
In the present article, we are only concerned with the case $q=2$, that is, the Thue--Morse sequence.)
The first auxiliary result, a Bombieri--Vinogradov type theorem, is an average result on the sets $A_\omega(y,z;d,j)$ and might also be of independent interest.

\begin{theorem}\label{thm:1}
Let $L\geq 1$ be an integer and $\omega=(\omega_0,\ldots,\omega_{L-1})\in\{0,1\}^L$.
Assume that
\[0<\delta_1\leq\delta_2<2/3.\]
There exist $\eta>0$ and a constant $C$ such that
\[
\sum_{D<d\leq 2D}\max_{\substack{y,z\\0\leq y\leq z\\z-y\leq x}}\max_{j\in\dZ}
\abs{
A_\omega(y,z;d,j)
-
\frac{z-y}{2^Ld}
}
\leq
C
x^{1-\eta}
\]
for all $x$ and $D$ such that $x\geq 1$ and $x^{\delta_1}\leq D\leq x^{\delta_2}$.
\end{theorem}
(Note that the maximum is well-defined by a finiteness argument. The same holds true for the subsequent results.)
This theorem differs in several aspects from~\cite[Corollary~3]{FM96}.
The most important novelty is the exponent $2/3$, which improves on the exponent $0.5924$.
Moreover, the left endpoint of the interval 
$[y,z)$ 
may be an arbitrary nonnegative real number (which works well in the sum-of-digits setting, but fails for prime numbers).
Finally, this theorem handles consecutive elements of arithmetic subsequences of the Thue--Morse sequence $\mathbf t$.
This latter feature necessitates a nontrivial lower bound for $D$,
since factors of length $2$ of $\mathbf t$ do not appear with frequency $1/4$,
therefore the contribution of $d=1$ would already be too large.

Setting $L=1$ and using the above-cited result (with the improved error term $x^{1-\eta}$) in order to handle small step lengths $d$, we obtain the following corollary.
\begin{cor}\label{cor:1}
For real $y\geq 0$ and integers $d\geq 1$ and $j\geq 0$ set
\[
A(y;d,j)=\abs{\{m<y:s(m)\equiv 0\bmod 2, m\equiv j\bmod d\}}.
\]
For all $\delta\in(0,2/3)$ there exist $\eta>0$ and $C$ such that
\[
\sum_{1\leq d\leq D}\max_{y\leq x}\max_{j\in\dZ}
\abs{  A(y;d,j) - \frac{y}{2d}  }    \leq Cx^{1-\eta}
\]
for $x\geq 1$ and $D=x^\delta$.
\end{cor}
A simple but interesting consequence of Theorem~\ref{thm:1} is the following result.
\begin{cor}
Every finite sequence on the symbols $\tO$ and $\tL$ appears as an arithmetic subsequence of the Thue--Morse sequence.
\end{cor}

An adaptation of the proof of \thmref{thm:1} yields the following Beatty sequence analogue.
\begin{theorem}\label{thm:2}
Let $L\geq 1$ be an integer, $\omega=(\omega_0,\ldots,\omega_{L-1})\in\{0,1\}^L$
and $0<\delta_1\leq\delta_2<2/3$.
There exist $\eta>0$ and $C$ such that
\[
\int_{D}^{2D}\max_{\substack{y,z\\0\leq y\leq z\\z-y\leq x}}\max_{\beta\geq 0}
\abs{  A_\omega(y,z;\alpha,\beta) - \frac{z-y}{2^L\alpha}  }
\,\mathrm d\alpha
\leq Cx^{1-\eta}
\]
for all $x$ and $D$ such that $x\geq 1$ and $x^{\delta_1}\leq D\leq x^{\delta_2}$.
\end{theorem}
As we announced in the introduction,
this theorem can be used to obtain \thmref{thm:main}.

The proofs of Theorems~\ref{thm:1} and~\ref{thm:2} are based on exponential sum estimates.
In order to establish \thmref{thm:1}, it is sufficient to prove the following proposition.
\begin{prop}\label{prp:1}
Let $L\geq 1$ be an integer, $a=(a_0,\ldots,a_{L-1})\in\{0,1\}^L$ and $a\neq (0,\ldots,0)$.
For real numbers $N,D\geq 1$ and $\xi$ set
\[
S_1 = S_1(N,D,\xi)
=
\sum_{D\leq d<2D}
\max_{j\geq 0}
\Biggl\lvert
  \sum_{n<N}
  \e\Biggl(\frac 12
    \sum_{\ell<L}
    a_\ell
    s\bigl((n+\ell)d+j\bigr)
  \Biggr)
  \e(n\xi)
\Biggr\rvert
.
\]
Let $0<\rho_1\leq\rho_2<2$.
There exists an $\eta>0$ and a constant $C$ such that
\begin{equation}\label{eqn:prp1_estimate}
\frac{S_1}{ND}
\leq
CN^{-\eta}
\end{equation}
holds for all $\xi\in\dR$ and all real numbers $N,D\geq 1$ satisfying
$N^{\rho_1}\leq D \leq N^{\rho_2}$.
\end{prop}
For $L=1$ this result intuitively states that for most step lengths $d$ slightly smaller than the square of the length of the sum,
we have a nontrivial estimate for sums over the Thue--Morse sequence on arithmetic progressions.

Analogously, \thmref{thm:2} is based on the following result.
\begin{prop}\label{prp:2}
Let $L\geq 1$ be an integer, $a=(a_0,\ldots,a_{L-1})\in\{0,1\}^L$ and $a\neq (0,\ldots,0)$.
For real numbers $D,N\geq 1$ and $\xi$ set
\[
\tilde S_1 = \tilde S_1(N,D,\xi)
=
\int_{D}^{2D}
\max_{\beta\geq 0}
\Biggl\lvert
  \sum_{n<N}
  \e\Biggl(
    \frac 12
    \sum_{\ell<L}
    a_\ell
    s\bigl(\lfloor(n+\ell)\alpha+\beta\rfloor\bigr)
  \Biggr)
  \e(n\xi)
\Biggr\rvert
\,\mathrm d \alpha
.
\]
Let $0<\rho_1\leq\rho_2<2$.
There exist $\eta>0$ and a constant $C$ such that
\begin{equation}\label{eqn:prp2_estimate}
\frac{\tilde S_1}{ND}
\leq
CN^{-\eta}
\end{equation}
holds for all real numbers $D,N\geq 1$ satisfying
$N^{\rho_1}\leq D \leq N^{\rho_2}$ and for all $\xi\in\dR$.
\end{prop}
The proofs of Propositions~\ref{prp:1} and~\ref{prp:2} in turn rely on an estimate for discrete Fourier coefficients related to the sum-of-digits function.
These Fourier coefficients have been used as an essential tool in the article~\cite{DMR2015} on the normality of the Thue--Morse sequence along the sequence of squares.
For nonnegative integers $d$ and $\lambda$, for sequences
$i:\dN\ra\dN$ and $a:\dN\ra\dZ$, where $a$ has finite support,
and for $h\in\dZ$ we define
\begin{equation}\label{eqn:G_definition}
G_\lambda^{i,a}(h,d)
=
\frac 1{2^\lambda}
\sum_{u<2^\lambda}
\e\Biggl(\frac 12
\sum_{\ell\in\dN}
a_\ell s_\lambda\bigl(u+\ell d+i_\ell\bigr)-\frac{hu}{2^\lambda}
\Biggr).
\end{equation}
The function $h\mapsto G_\lambda^{i,a}(h,d)$ is the discrete Fourier transform of the $2^\lambda$-periodic sequence
\[
  n \mapsto \e\Biggl(\frac 12\sum_{\ell\geq 0} a_\ell s_\lambda\bigl(n + \ell d + i_\ell\bigr)\Biggr).
\]
We have the following important technical estimate for these Fourier terms.
\begin{prop}\label{prp:uniform_fourier}
Let $L\geq 1$ be an integer and choose $m\geq 5$ such that $2^{m-5}\leq L<2^{m-4}$.
Assume that
$(a_\ell)_{\ell\in\dZ}$ is a sequence such that
$a_0=1$, $a_1,\ldots,a_{L-1}\in\{0,1\}$
and $a_\ell=0$ for 
$\ell\not\in[0,L)$. 
For $r\geq 1$ and $\ell\geq 0$ let $b^r_\ell=a_{\ell-r}-a_\ell$.
There exist $\eta>0$ and $C$
such that for all $\lambda\geq 0$ and $r\geq 1$ satisfying
$2m\leq\nu_2(r)\leq \lambda/4$,
and for all
sequences $(i_\ell)_{\ell\in\dZ}$ satisfying
$i_0=0$ and $0\leq i_{\ell+1}-i_\ell\leq 1$ for $0\leq \ell<L+r-1$ we have
\[
\frac 1{2^\lambda}\sum_{d<2^\lambda}\max_{h<2^\lambda}
\abs{  G_\lambda^{i,b^r}(h,d)  }^2    \leq    C2^{-\eta\lambda}.
\]
\end{prop}
This result differs from \cite[Proposition~1]{DMR2015} in two aspects.
First, the maximum over~$h$ is \emph{inside} the sum over~$d$;
second, the sequence $b^r$ consists of two identical blocks (modulo $2$) spaced by $r$.
The constant as well as the exponent do not depend on the shift $r$,
which will allow us to prove the quantitative normality result as stated in \thmref{thm:main}.
\begin{rem}
Combining \propref{prp:uniform_fourier} with the method of proof from~\cite{S2015}, we could obtain a quantitative version of \thmref{thm:S2015} that differs from our main theorem only in the (worse) bound $c<4/3$.
\end{rem}

Finally, we state the following result, which allows replacing $\floor{n^c}$ by $\floor{n\alpha+\beta}$.
\begin{prop}\label{prp:PS_via_beatty_2}
Assume that $\varphi:\dN\ra\Omega$ is a function into a finite set $\Omega$ and assume that $\omega=(\omega_0,\ldots,\omega_{L-1})\in \Omega^L$, where $L\geq 1$ is an integer.
We write $f(x)=x^c$, where $1<c<2$ is a real number.
Let $\delta\in[0,1]$ and define, for real $N,K>0$,
\begin{multline}\label{eqn:essential_integral}
J(N,K)
=
\frac 1{f'(2N)-f'(N)}
\int_{f'(N)}^{f'(2N)}
\max_{f(N)<\beta\leq f(2N)}
\biggl\lvert
\frac 1K
\bigl\lvert\bigl\{
n<K:
\\
\varphi\bigl(\lfloor(n+\ell)\alpha+\beta\rfloor\bigr)=\omega_\ell\text{ for }0\leq \ell<L
\bigr\}\bigr\rvert
-\delta
\biggr\rvert
\,\mathrm d\alpha
.
\end{multline}
There exists a constant $C$ such that for all $N\geq 2$ and $K>0$ we have
\begin{multline}\label{eqn:block_substitution_rule}
    \biggl\lvert
        \frac 1N
        \bigl\lvert\bigl\{
           n\in (N,2N]:
           \varphi\bigl(\floor{f(n+\ell)}\bigr)=\omega_\ell
           \text{ for }0\leq \ell<L
        \bigr\}\bigr\rvert
      -
        \delta
    \biggr\rvert
\\
  \leq    C \biggl(  f''(N)K^2 + \frac{(\log N)^2}{K} + J(N,K)  \biggr).
\end{multline}
\end{prop}
Results on the distribution of values of an arithmetic function $\varphi:\dN\ra\Omega$ on \emph{Beatty} sequences can therefore be used for proving statements concerning $\varphi$ on \emph{Piatetski-Shapiro} sequences $n\mapsto \floor{n^c}$, at least in cases where the shift $\beta$ does not cause problems.
(This is the case for our problem concerning $\bt$, however, this proposition cannot be used for the original Piatetski-Shapiro problem, since our knowledge on primes in short intervals is not sufficient.)
\propref{prp:PS_via_beatty_2} is a modification of~\cite[Proposition~1]{S2014}, which, together with the statement by Fouvry and Mauduit~\cite[Th\'eor\`eme~3 and inequality~(1.5)]{FM96} asserting that
\begin{align}\label{eqn:FM}
\int_0^1\prod_{j<k}\bigl\lvert\sin\bigl(2^j\pi\theta\bigr)\bigr\rvert\,\mathrm d\theta    \sim    \kappa\lambda^k
\end{align}
for some $\kappa\in\dR$ and some $\lambda\in (0.6543,0.6632)$,
enabled the second author to obtain simple normality of $n\mapsto \bt(\floor{n^c})$ for $c\leq 1.42$.

The plan of the paper is as follows.
In Section~\ref{sec:lemmas} we state a number of lemmas.
In Section~\ref{sec:reduction} we show how to prove Theorems~\ref{thm:1}, \ref{thm:2} and~1 from Propositions~\ref{prp:1} and~\ref{prp:2}.
Section~\ref{sec:proof_prp1} is concerned with the proof of \propref{prp:1},
while Section~\ref{sec:proof_prp2}, which is shorter,
proves \propref{prp:2} in a way that is to a large extent analogous.
(This section is shorter because some parts that have been treated in detail in the first proof have been left out.
We also note that it would be possible to unify to a large extent the proofs of Propositions~\ref{prp:1} and~\ref{prp:2} by rewriting some sums as integrals with respect to some measure.
However, we refrained from doing so since we wanted to keep the presentation clear.)

The last two sections are dedicated to the proofs of Propositions~\ref{prp:uniform_fourier} and~\ref{prp:PS_via_beatty_2}.
\section{Lemmas}\label{sec:lemmas}
We begin with the following elementary facts about the functions $\floor{\cdot}$, $\norm{\cdot}$ and $\ang{\cdot}$, where $\ang{x}=\floor{x+1/2}$ (the ``nearest integer'' to $x$).
The (easy) proofs are left to the reader.
\begin{lem}\label{lem:fractional_part_facts}
Let $a,b\in\dR$ and $n\in\dN$.
\begin{enumerate}[(i)]
\item \label{item:fpf_1} If $\norm{a}<\varepsilon$ and $\norm{b}\geq \varepsilon$, then
$\floor{a+b}=\ang{a}+\floor{b}$.
\item \label{item:fpf_2} $\norm{na}\leq n\norm{a}$.
\item \label{item:fpf_3} If $\norm{a}<\varepsilon$ and $2n\varepsilon<1$, then
$\ang{na}=n\ang{a}$.
\end{enumerate}
\end{lem}
An essential tool in our proofs is the following generalization of van der Corput's inequality (see~\cite[Lemme 17]{MR2009}).
\begin{lem}\label{lem:vdc}
Let $I$ be a finite interval containing $N$ integers and
let $a_n$ be a complex number for $n\in I$.
For all integers $K\geq 1$ and $R\geq 1$ we have
\[
  \abs{\sum_{n\in I}a_n}^2
\leq
  \frac{N+K(R-1)}R
  \sum_{\lvert r\rvert<R}\biggl(1-\frac {\lvert r\rvert}R\biggr)
  \sum_{\substack{n\in I\\n+Kr\in I}}a_{n+Kr}\overline{a_n}
.
\]
In particular, the right hand side is a nonnegative real number.
\end{lem}
The following simple lemma will help us to remove the expression $\lfloor n\alpha+\beta\rfloor$, which appears in our calculations via linear approximation.
It introduces the \emph{discrepancy} of a sequence (modulo~$1$) instead.
\begin{lem}\label{lem:f_discrepancy}
Let $J$ be an interval in $\dR$ containing $N$ integers
and let $\alpha$ and $\beta$ be real numbers.
Assume that $t,T,k$ and $K$ are integers such that $0\leq t<T$ and $0\leq k<K$.
Then
\begin{multline*}
\biggl\lvert
  \biggl\{
    n\in J:
    \frac tT\leq \{n\alpha+\beta\}<\frac{t+1}T,
    \lfloor n\alpha+\beta \rfloor\equiv k\bmod K
  \biggr\}
\biggr\rvert
\\=
\frac N{KT} + O\Bigl(ND_N\Bigl(\frac \alpha K\Bigr)\Bigr),
\end{multline*}
where
\[ 
D_N(\alpha)=\sup_{\substack{0\leq x\leq 1\\y\in\dR}}\Biggl\lvert \frac 1N\sum_{n<N}c_{[0,x)+y+\dZ}(n\alpha)-x  \Biggr\rvert. 
\]
\end{lem}
\begin{proof}
We set                                     
$I=\bigl[(Tk+t)/(KT),(Tk+t+1)/(KT)\bigr)$, 
which is a subinterval of                  
$[0,1)$                                    
of length $1/(KT)$.
The two conditions
$t/T\leq \{n\alpha+\beta\}<(t+1)/T$ and
$\lfloor n\alpha+\beta\rfloor\equiv k\bmod K$
are satisfied if and only if
$\{n\alpha/K+\beta/K\}\in I$
and the lemma follows by inserting the definition of $D_N(\alpha)$.
\end{proof}
In order to handle the discrepancy (of $n\alpha$-sequences) thus introduced, we will use average results as in the following lemma.
\begin{lem}\label{lem:mean_discrepancy}
Let $J$ be a finite interval in $\dR$ containing $N$ integers.
Then
\begin{equation}\label{eqn:mean_geometric_sum}
\sum_{k<2^\rho}
\Biggl\lvert
  \sum_{j\in J}
  \e\biggl(\frac{jmk}{2^\rho}\biggr)
\Biggr\rvert
\ll
2^{\nu_2(m)}N+2^\rho \logp N
\end{equation}
for all integers $\rho\geq 0$ and $m\neq 0$.
For integers $\mu\geq 0$ and $N\geq 1$ we have
\begin{equation}\label{eqn:mean_discrepancy_sum}
\sum_{d<2^{\mu}}
D_N\biggl(\frac d{2^\mu}\biggr)
\ll
\frac{N+2^\mu}{N}\bigl(\logp N\bigr)^2.
\end{equation}
Moreover, the estimate
\begin{equation}\label{eqn:mean_discrepancy_int}
\int_0^1 D_N(\alpha)\,\mathrm d\alpha
\ll
\frac{\bigl(\logp N\bigr)^2}{N}
\end{equation}
holds. The implied constants in these three estimates are absolute.
\end{lem}
\begin{proof}
We prove the first claim.
The estimate is trivial for $N\leq 1$. We assume therefore that $N\geq 2$.
Let $0<a\leq b$. We have
\begin{align*}
\sum_{a\leq k<b}\frac 1k
&=
\sum_{\lfloor a\rfloor+1\leq k<\lfloor b\rfloor}\frac 1k+O(1)
=
\log(\lfloor b\rfloor)-\log(\lfloor a\rfloor+1)+O(1)
\\&\leq
\log b-\log a+O(1)
.
\end{align*}
Therefore we get for all integers $\rho\geq 1$
\begin{align*}
\sum_{k<2^\rho}
\min\Bigl\{N,\norm{k/2^\rho}^{-1}\Bigr\}
&=
2
\sum_{k<2^{\rho-1}}
\min\Bigl\{N,\lvert k/2^\rho\rvert^{-1}\Bigr\}
\\
&\ll
N\bigl\lvert\bigl\{k<2^{\rho-1}:k<2^\rho/N\bigr\}\bigr\rvert
+
2^\rho\sum_{2^\rho/N\leq k<2^{\rho-1}}\frac 1k
\\
&\ll
N(1+2^\rho/N)+2^\rho\bigl(1+\log 2^\rho-\log(2^\rho/N)\bigr)
\\
&\ll
N+2^\rho \logp N.
\end{align*}
This estimate is also valid for $\rho=0$.
Let $2^\eta\mid m$ and $2^{\eta+1}\nmid m$, that is, $\nu_2(m)=\eta$.
If $\eta\leq \rho$, we have
\begin{align*}
\sum_{k<2^\rho}
\min\Bigl\{N,\norm{km/2^\rho}^{-1}\Bigr\}
&=
2^\eta
\sum_{k<2^{\rho-\eta}}
\min\Bigl\{N,\norm{k/2^{\rho-\eta}}^{-1}\Bigr\}
\\&\ll
2^\eta N+2^\rho\logp N.
\end{align*}
Note that this estimate holds trivially for $\eta>\rho$.
The statement~\eqref{eqn:mean_geometric_sum} follows therefore from the inequality
\[
\Biggl\lvert\sum_{j\in J}\e\bigl(jmk/2^\rho\bigr)\Biggr\rvert
\leq \min\Bigl\{N,\norm{km/2^\rho}^{-1} \Bigr\}.
\]
In order to prove the first result on the average discrepancy,
we use the Erd\H{o}s--Tur\'an inequality and~\eqref{eqn:mean_geometric_sum} and obtain
\begin{align*}
N \sum_{d<2^{\mu}} D_N\biggl(\frac d{2^\mu}\biggr)
&\ll
2^\mu
+
\sum_{1\leq h\leq N}
\frac 1h
\sum_{d<2^\mu}
\Biggl\lvert
\sum_{n<N}
\e\biggl(\frac{hnd}{2^\mu}\biggr)
\Biggr\rvert
\\&\ll
2^\mu
+
\sum_{\rho\leq \frac{\log N}{\log 2}}
\sum_{\substack{1\leq h\leq N\\\nu_2(h)=\rho}}
\frac 1h
\bigl(
2^\rho N
+
2^\mu\logp N
\bigr)
\\&\ll
2^\mu+
\logp N
\sum_{\rho\leq \frac{\log N}{\log 2}}
\frac 1{2^\rho}
\bigl(
2^\rho N
+
2^\mu\logp N
\bigr)
\\&\ll
\bigl(N+2^\mu\bigr)\bigl(\logp N\bigr)^2.
\end{align*}
The proof of the last statement is analogous.
\end{proof}
The following lemma concerning the discrete Fourier transform can easily be proved using orthogonality relations.
\begin{lem}\label{lem:correlation_fourier}
Assume that $M\geq 1$ is an integer and that $f:\dZ\ra\dC$ is an $M$-periodic function.
Then
\begin{equation}\label{eqn:correlation_rep}
\frac 1M
\sum_{n<M}
f(n+t)\overline{f(n)}
=
\sum_{h<M}
\bigl\lvert\hat f(h)\bigr\rvert^2\e\p{ht/M}
,
\end{equation}
where
\[    \hat f(h) = \frac 1M\sum_{u<M}f(u)\e\p{-hu/M}.    \]
\end{lem}
We also need the following \emph{carry propagation lemma} for the sum-of-digits function.
Statements of this type were used in the articles~\cite{MR2009,MR2010} by
Mauduit and Rivat on the sum of digits of primes and squares.
\begin{lem}\label{lem:carry}
Let $r,L,N,\lambda$ be nonnegative integers and $\alpha>0,\beta\geq 0$ real numbers.
Assume that $I$ is an interval containing $N$ integers.
Then
\begin{align*}
  \bigl\lvert
    \bigl\{
      n\in I:
      \exists \ell\in[0,L)\text{ such that }
      s(\lfloor(n+\ell+r)\alpha+\beta\rfloor)-s(\lfloor(n+\ell)\alpha+\beta\rfloor)
\\&&
\mathmakebox[0pt][r]{
\neq
      s_\lambda(\lfloor(n+\ell+r)\alpha+\beta\rfloor)-s_\lambda(\lfloor(n+\ell)\alpha+\beta\rfloor)
      \bigr\}
  \bigr\rvert
}
\\&&
\mathmakebox[0pt][r]{
\leq
  (r+L)(N\alpha/2^\lambda+2).
}
\end{align*}
\end{lem}
\begin{proof}
Let $E=(r+L)\alpha$. The statement is trivial for $E\geq 2^\lambda$.
We assume therefore that $E<2^\lambda$.
Moreover, we may assume that $L\geq 1$, since the estimate is trivial for $L=0$.
We first note that if
\begin{equation}\label{eqn:carry_sufficient} 
n\alpha+\beta \in [0,2^\lambda-E)+2^\lambda\dZ, 
\end{equation}
then
\begin{multline*}
s(\lfloor(n+\ell+r)\alpha+\beta\rfloor)-s(\lfloor(n+\ell)\alpha+\beta\rfloor)
\\=s_\lambda(\lfloor(n+\ell+r)\alpha+\beta\rfloor)-s_\lambda(\lfloor(n+\ell)\alpha+\beta\rfloor)
\end{multline*}
for all $\ell<L$.
This follows easily by studying the binary representation of the arguments:
if hypothesis~\eqref{eqn:carry_sufficient} is satisfied, then
$(n+0)\alpha+\beta,\ldots,(n+L-1+r)\alpha+\beta$ are contained in an interval 
$[k2^\lambda,(k+1)2^\lambda-1)$,                                              
therefore the digits of
$\lfloor(n+\ell+r)\alpha+\beta\rfloor$ and $\lfloor(n+\ell)\alpha+\beta\rfloor$
with indices $\geq \lambda$ are the same, for all $\ell<L$.
It remains to count the number of exceptions to~\eqref{eqn:carry_sufficient}.

For $k\in\dZ$ let $a_k=\min\{n:k2^\lambda\leq n\alpha+\beta\}$
and $b_k=\min\{n:(k+1)2^\lambda-E\leq n\alpha+\beta\}$.
Then for $a_k\leq n<b_k$ we have 
$n\alpha+\beta\in[0,2^\lambda-E)+2^\lambda\dZ$. 
It is therefore sufficient to count the number of $n\in I$ such that
$b_k\leq n<a_{k+1}$ for some $k$.

Clearly we have $a_{k+1}-b_k=r+L$.
Assume that 
$I=[a,b)$ 
and choose $k$ in such a way that $k2^\lambda\leq a\alpha+\beta<(k+1)2^\lambda$.
Then $a_k\leq a$.
Moreover
$(b-1)\alpha+\beta
<
a\alpha+\beta+(b-a-1)\alpha
<
(k+1+(b-a-1)\alpha/2^\lambda)2^\lambda
<
(k+\lfloor N\alpha2^{-\lambda}\rfloor+2)2^\lambda
$,
therefore
$b-1<a_{k+\lfloor N\alpha2^{-\lambda}\rfloor+2}$.
It follows that the exceptional indices $n$ are contained in one of $\lfloor N\alpha2^{-\lambda}\rfloor+2$ intervals of length $r+L$.
\end{proof}
The following standard lemma allows us to extend the range
of a summation in exchange for a controllable factor.
\begin{lem}\label{lem:vinogradov}
Let $x\leq y\leq z$ be real numbers and $a_n\in\dC$ for $x\leq n<z$.
Then
\begin{equation*}
    \abs{\sum_{x\leq n<y}a_n}
  \leq
    \int_0^1{
      \min\left\{y-x+1,\norm{\xi}^{-1}\right\}
      \abs{\sum_{x\leq n<z}a_n\e\p{n\xi}}
    }
    \,\mathrm d \xi
.
\end{equation*}
\end{lem}
\begin{proof}
Since
$\int_0^1\e\p{k\xi}\,\mathrm d \xi=\delta_{k,0}$ for $k\in\dZ$
we have
\begin{equation*}
    \sum_{x\leq n<y}a_n
  =
    \sum_{x\leq n<z}a_n\sum_{x\leq m<y}\delta_{n-m,0}
  =
    \int_0^1
    \sum_{x\leq m<y}\e\p{-m\xi}
    \sum_{x\leq n<z}a_n\e\p{n\xi}
    \,\mathrm d \xi,
\end{equation*}
from which the statement follows.
\end{proof}
Let $\mathcal F_n$ be the \emph{Farey series of order} $n$, by which we understand the set of rational numbers $p/q$ such that $1\leq q\leq n$.
It is easy to see that each $a\in \mathcal F_n$ has two \emph{neighbours} $a_L,a_R\in \mathcal F_n$, satisfying $a_L<a<a_R$ and $(a_L,a)\cap \mathcal F=(a,a_R)\cap\mathcal F=\emptyset$.
We have the following elementary lemma concerning this set, which follows from the theorems in chapter~3 of the book~\cite{HW54} by Hardy and Wright.
\begin{lem}\label{lem:farey}
Assume that $a/b$, $c/d$ are reduced fractions such that $b,d>0$ and $a/b<c/d$.
Then $a/b<(a+c)/(b+d)<c/d$.
If $a/b$ and $c/d$ are neighbours in $\mathcal F_n$, then
$bc-ad=1$ and $b+d>n$,
moreover
\[    (a+c)/(b+d)-a/b < \frac 1{bn}    \]
and
\[    c/d-(a+c)/(b+d) < \frac 1{dn}.    \]
\end{lem}
We will also use the large sieve inequality, which we state here in the form provided by Selberg (see for example Montgomery~\cite[Theorem 3]{M78}).
\begin{lem}[Selberg]\label{lem:large_sieve}
Let $N\geq 1,R\geq 1$ and $M$ be integers,
$\alpha_1,\ldots,\alpha_R\in\dR$ and
$a_{M+1},\ldots,$ $a_{M+N}\in\dC$.
Assume that $\norm{\alpha_r-\alpha_s}\geq \delta$ for $r\neq s$, where $\delta>0$.
Then
\[
\sum_{r=1}^R
\abs{
\sum_{n=M+1}^{M+N}
a_n\e(n\alpha_r)
}^2
\leq
(N-1+\delta^{-1})
\sum_{n=M+1}^{M+N}
\abs{a_n}^2
.
\]
\end{lem}
\section{Reduction of the problem}\label{sec:reduction}
Before proving Propositions~\ref{prp:1} and~\ref{prp:2}, we want to show that these results imply Theorems~\ref{thm:1} and~\ref{thm:2} respectively.
\subsection{Reducing Theorems~\ref{thm:1} and~\ref{thm:2} to Propositions~\ref{prp:1} and~\ref{prp:2}}
We only deduce \thmref{thm:1} from \propref{prp:1}.
The second implication can be shown in an analogous way.

Assume that the statement of \propref{prp:1} holds for some $\rho_1$, $\rho_2$ such that $0<\rho_1\leq\rho_2$.
Note that we allow $\rho_i\geq 2$ to keep the proof more general.
We want to show that in \thmref{thm:1} we may choose
$\delta_1=\rho_1/(\rho_1+1)$ and $\delta_2=\rho_2/(\rho_2+1)$.
Choosing $\rho_2$ close to $2$, justified by \propref{prp:1},
we see that $\delta_2$ approaches $2/3$.

For real numbers $x,D\geq 1$ we define
\begin{equation}\label{eqn:S0_definition}
S_0 = S_0(x,D)
=
\sum_{D<d\leq 2D}
\max_{\substack{y,z\\0\leq y\leq z\\z-y\leq x}}
\max_{j\in\dZ}
\abs{A_\omega(y,z;d,j)-\frac {z-y}{2^Ld}}
.
\end{equation}
Let $x\geq 1$ and $D$ be real numbers such that
$x^{\delta_1}\leq D\leq x^{\delta_2}$.
We rewrite the difference appearing in $S_0$ to exponential sums, using orthogonality relations:
\begin{align*}
\hspace{1em}&\hspace{-1em} A_\omega(y,z;d,j) -\frac {z-y}{2^Ld}
\\&=
\sum_{\substack{m,n\\y\leq m<z\\m=nd+j}}
\left\{\begin{array}{ll}1&\text{if }s((n+\ell)d+j)\equiv \omega_\ell\bmod 2\text{ for }\ell<L\\0&\text{otherwise}\end{array}\right\}
-\frac {z-y}{2^Ld}
\\&=
\sum_{n<\frac{z-y}{d}}
\left\{\begin{array}{ll}1&\text{if }s((n+\ell)d+j+\floor{(y-j)/d}d)\equiv \omega_\ell\bmod 2\text{ for }\ell<L\\0&\text{otherwise}\end{array}\right\}
\\&\qquad-
\frac {z-y}{2^Ld}
+O(1)
\\&=
\frac 1{2^L}
\sum_{\substack{a\in\{0,1\}^L\\a\neq (0,\ldots,0)}}
\e\p{-\frac 12(a_0\omega_0+\cdots+a_{L-1}\omega_{L-1}) }
\\&\qquad
\times
\sum_{n<(z-y)/d}
\e\p{\frac 12\sum_{\ell<L}
a_\ell s((n+\ell)d+j+\floor{(y-j)/d}d)
}
+O(1).
\end{align*}
It follows that
\[
S_0
\leq
\frac 1{2^L}
\sum_{\substack{a\in\{0,1\}^L\\a\neq (0,\ldots,0)}}
\sum_{D\leq d<2D}
\max_{u\leq x}
\max_{j\geq 0}
\abs{
  \sum_{n<u/d}
  \e\p{\frac 12\sum_{\ell<L}a_\ell s((n+\ell)d+j)}
}
+O(D)
.
\]
In order to dispose of the maximum over $u$, we apply \lemref{lem:vinogradov}.
We exchange the appearing integral with the maximum over $j$ and with the sum over $d$:
\begin{multline}\label{eqn:S0_estimate}
S_0(x,D)
\leq
\frac 1{2^L}
\sum_{\substack{a\in\{0,1\}^L\\a\neq (0,\ldots,0)}}
\int_0^1
\min\left\{\frac xD+1,\norm{\xi}^{-1}\right\}
\\
\times
\sum_{D\leq d<2D}
\max_{j\geq 0}
\abs{
  \sum_{n<x/D}
  \e\p{\frac 12\sum_{\ell<L}a_\ell s((n+\ell)d+j)}
  \e\p{n\xi}
}
\,\mathrm d\xi
+O(D)
.
\end{multline}
\propref{prp:1} therefore implies that there exist a constant $C$ and an exponent $\eta>0$ such that for all $a_0,\ldots,a_{L-1}\in\{0,1\}$, not all equal to zero, for all real numbers $x,D\geq 1$ satisfying
\[    \p{\frac xD}^{\rho_1}\leq D\leq \p{\frac xD}^{\rho_2},    \]
and for all $\xi\in\dR$ we have
\begin{equation}\label{eqn:critical_estimate}
\sum_{D\leq d<2D}
\max_{j\geq 0}
\abs{
  \sum_{n<x/D}
  \e\p{\frac 12\sum_{\ell<L}a_\ell s((n+\ell)d+j)}
  \e\p{n\xi}
}
\leq C D \p{\frac xD}^{1-\eta}
.
\end{equation}
The condition on $D$ can be rewritten as
$x^{\delta_1}\leq D\leq x^{\delta_2}$.
By the estimate
\[
\int_0^1
\min\left\{A,\norm{\xi}^{-1}\right\}\,\mathrm d \xi
=
2\p{\int_0^{1/A}A\,\mathrm d \xi+\int_{1/A}^{1/2}\xi^{-1}\,\mathrm d\xi}
\ll \log A,
\]
which holds for $A\geq 2$,
and by~\eqref{eqn:S0_estimate} and~\eqref{eqn:critical_estimate},
there exists some $\eta_1>0$ and constants $C$ and $C_1$ such that for all $x,D\geq 1$ satisfying $x^{\delta_1}\leq D\leq x^{\delta_2}$ we have
\[
S_0(x,D)
\leq
Cx\p{\frac Dx}^\eta\logp x
+O(D)
\leq C_1x^{1-\eta_1},
\]
which proves the theorem.
\subsection{Proving \thmref{thm:main} from \thmref{thm:2} and \propref{prp:PS_via_beatty_2}}\label{subsec:nc}
We show the more general implication that if the statement of \thmref{thm:2} holds for
some $0<\delta_1\leq\delta_2<1$, then we may choose $c<2/(2-\delta_2)$ in \thmref{thm:main}.
Choosing $\delta_2$ close to $2/3$ yields the desired statement.
We have to find an estimate for $J(N,K)$ defined in~\eqref{eqn:essential_integral}.
Therefore we calculate:
\begin{align*}
\hspace{4em}&\hspace{-4em}\bigl\lvert\{n<K:s(\floor{(n+\ell)\alpha+\beta})\equiv \omega_\ell\bmod 2\text{ for }\ell<L\}\bigr\rvert\\
&=
\bigl\lvert\{\floor{\beta}\leq m<\floor{K\alpha+\beta}:\exists n\in\dZ:m=\lfloor n\alpha+\beta\rfloor,
\\&&\mathmakebox[0pt][r]{
s(\floor{(n+\ell)\alpha+\beta})\equiv\omega_\ell\bmod 2\text{ for }\ell<L\}\bigr\rvert}
\\&=
\bigl\lvert\{\beta\leq m<K\alpha+\beta:\exists n\in\dZ:m=\lfloor n\alpha+\beta\rfloor,
\\&&\mathmakebox[0pt][r]{
s(\floor{(n+\ell)\alpha+\beta})\equiv\omega_\ell\bmod 2\text{ for }\ell<L\}\bigr\rvert+O(1)}
\\&=
A_\omega(\beta,K\alpha+\beta;\alpha,\beta)+O(1)
.
\end{align*}
We use the definition~\eqref{eqn:essential_integral}, where we set $\delta=2^{-L}$,
and define $D=f'(N)$. Noting that $f'(2N)=2^{c-1}D\leq 2D$, we obtain
\begin{align*}
J(N,K)
&\leq
\frac 1{f'(2N)-f'(N)}
\frac 1K
\int_{f'(N)}^{f'(2N)}
\max_{\beta\geq 0}
\biggl\lvert A_\omega(\beta,K\alpha+\beta;\alpha,\beta)
\\&&\mathmakebox[0pt][r]{
-\frac{K\alpha+\beta-\beta}{2^L\alpha}\biggr\rvert
\,\mathrm d\alpha
+O(1/K)}
\\&\leq
\frac 1{(2^{c-1}-1)DK}
\int_D^{2D}
\max_{\substack{y,z\\0\leq y\leq z\\z-y\leq 2DK}}\max_{\beta\geq 0}
\biggl\lvert A_\omega(y,z;\alpha,\beta)
\\&&\mathmakebox[0pt][r]{
-\frac{z-y}{2^L\alpha}
\biggr\rvert
\,\mathrm d\alpha
+O(1/K)}
.
\end{align*}
It follows from \thmref{thm:2} that for
$(2DK)^{\delta_1}\leq D\leq(2DK)^{\delta_2}$,
that is, for $\tfrac 12D^{1/\delta_2-1}\leq K\leq \tfrac 12D^{1/\delta_1-1}$, we have
\[    J(N,K) \leq \frac C{DK} (2DK)^{1-\eta}+O(1/K)    \]
for some $\eta>0$ and $C$ depending on $c$ and $L$.
Setting $K=\tfrac 12D^{1/\delta_2-1}$, we obtain
\[
J(N,K)\leq CD^{-\eta/\delta_2}+2D^{1-1/\delta_2}
.
\]
By \propref{prp:PS_via_beatty_2} we get
\begin{align*}
\hspace{4em}&\hspace{-4em}
    \abs{
        \frac 1N
        \bigl\lvert\bigl\{
          n\in (N,2N]:
          s(\floor{(n+\ell)^c})=\omega_\ell
          \text{ for }0\leq \ell<L
        \bigr\}\bigr\rvert
      -
        \frac 1{2^L}
    }
\\&\leq
  C_1\p{
      f''(N)K^2
    +
      \frac{(\log N)^2}{K}
    +
      J(N,K)
  }
\\&\leq
  C_2\p{
      N^{c-2+2(c-1)(1/\delta_2-1)}
    +
      \frac{(\log N)^2}{N^{(c-1)(1/\delta_2-1)}}
    +
      N^{-\eta(c-1)/\delta_2}
  }
.
\end{align*}
All of the occurring exponents of $N$ are negative by the conditions $c<2/(2-\delta_2)$
and $0<\delta_2<1$, which proves \thmref{thm:main}.
\section{Proof of \propref{prp:1}}\label{sec:proof_prp1}
Assume that $L\geq 1$ is an integer, $a=(a_0,\ldots,a_{L-1})\in\{0,1\}^L$ and that $a_\ell=1$ for some $\ell$.
It is easy to see, using the shift $j$, that we may assume $a_0=1$.
We also assume that $N\geq 1$ is an integer;
the general statement follows from the estimate $S_1(N,D,\xi)-S_1(\floor{N},D,\xi)\ll D$.
Moreover, it is sufficient to prove the statement
\[    \frac{S_1(N,2^\nu,\xi)}{N2^\nu}\leq CN^{-\eta}    \]
for all integers $N,\nu\geq 1$ and real numbers $D\geq 1$ such that
$N^{\rho_1}\leq D\leq N^{\rho_2}$ and $D<2^\nu\leq 2D$.
This can be seen by considering sums (in $d$) over the intervals 
$[2^{\nu-1},2^\nu)$ and $[2^\nu,2^{\nu+1})$ 
and using the estimate
$S_1(N,2^{\nu-1},\xi)\leq S_1(N,2^\nu,\xi)$,
which follows from the identity $s(2m)=s(m)$.
Choose $\eta$ and $C$ according to \propref{prp:uniform_fourier}.
Moreover, let $\tau=2m$, where $2^{m-5}\leq L<2^{m-4}$, and $\lambda\geq 0$.
By the Cauchy--Schwarz inequality we have
\begin{equation}\label{eqn:S1_S2}
\abs{S_1(N,2^\nu,\xi)}^2
\leq
2^\nu
\sum_{2^\nu\leq d<2^{\nu+1}}
\max_{j\geq 0}
S_2(N,d,j,\xi)
,
\end{equation}
where
\[
S_2 = S_2(N,d,j,\xi)
=
\abs{
  \sum_{n<N}{
    \e\p{
      \frac 12
      \sum_{\ell<L}
      a_\ell
      s((n+\ell)d+j)
    }
    \e(n\xi)
  }
}^2
.
\]
We apply \lemref{lem:vdc} (the generalized inequality of van der Corput) with $K=2^\tau$.
Let $R\geq 1$ be an integer. Then
\begin{multline*}
S_2
\leq
\frac{N+2^\tau(R-1)}{R}
\sum_{\abs{r}<R}
\p{1-\frac{\abs{r}}{R}}
\e(r2^\tau\xi)
\\
\times
\sum_{0\leq n,n+r2^\tau<N}
\e\p{\frac 12\sum_{\ell<L}a_\ell\bigl(s((n+\ell+r2^\tau)d+j)-s((n+\ell)d+j)\bigr)}
.
\end{multline*}
Using \lemref{lem:carry} and treating the summand $r=0$ separately,
moreover omitting the condition $0\leq n+r2^\tau<N$,
we obtain for all $\lambda\geq 0$
\begin{multline*}
S_2
\ll
O\p{\frac{N^2}R+NR2^\tau+N^2\frac{R2^\tau d}{2^\lambda}}
\\+
\frac NR
\sum_{1\leq r<R}
\abs{
  \sum_{n<N}
  \e\p{
    \frac 12
    \sum_{\ell<L}
    a_\ell
    \bigl(
      s_\lambda((n+\ell+r2^\tau)d+j)
      -
      s_\lambda((n+\ell)d+j)
    \bigr)
  }
}
\end{multline*}
with an implied constant depending only on the block length $L$.
Note that we also replaced $N+(R-1)2^\tau$ by $N$;
this is clearly admissible if $R2^\tau\leq N$, otherwise we use the trivial estimate $\abs{S_2}\leq N^2$ and note the presence of the error term $O(NR 2^\tau)$.

We set $a_\ell$ to $0$ for $\ell\not\in\{0,\ldots,L-1\}$
and define $b^r_\ell=a_{\ell-r2^\tau}-a_\ell$.
Applying the Cauchy--Schwarz inequality twice and using~\eqref{eqn:S1_S2} gives
\begin{multline}\label{eqn:S1_estimate}
\abs{S_1(N,2^\nu,\xi)}
^2
\ll
(2^\nu N)^2 O\p{\frac 1R+\frac{R2^\tau 2^\nu}{2^\lambda}+\frac {R2^\tau}N}
\\+
2^{3\nu/2}N
\p{
\frac 1R
\sum_{1\leq r<R}
S_3(N,r,\lambda,\nu)
}^{1/2}
,
\end{multline}
where
\begin{equation}\label{eqn:S3_definition}
S_3(N,r,\lambda,\nu)
=
\sum_{2^\nu\leq d<2^{\nu+1}}
\max_{j\geq 0}
\abs{
  \sum_{n<N}
  \e\p{
    \frac 12
    \sum_{\ell\in\dZ}
    b^r_\ell
    s_\lambda((n+\ell)d+j)
  }
}^2
.
\end{equation}
Applying Lemma \ref{lem:vdc} for $K=2^\mu$, we obtain for all integers $M\geq 1$
and $\mu\geq 0$
\begin{multline}\label{eqn:S3_summand_estimate}
\abs{
  \sum_{n<N}
  \e\p{
    \frac 12
    \sum_{\ell\in\dZ}
    b^r_\ell
    s_\lambda((n+\ell)d+j)
  }
}^2
\leq
\frac{N+2^\mu(M-1)}{M}
\sum_{\abs{m}<M}
\p{1-\frac{\abs{m}}M}
\\
\times
\sum_{0\leq n,n+m2^\mu<N}
\e\p{
  \frac 12\sum_{\ell\in\dZ}
  b^r_\ell
  \bigl(
      s_{\mu,\lambda}((n+\ell+m2^\mu)d+j)
    -
      s_{\mu,\lambda}((n+\ell)d+j)
  \bigr)
}
\\
\ll
N\abs{S_4}
+
2^\mu MN
,
\end{multline}
where
\begin{align*}
S_4 &= S_4(N,M,d,j,r,\mu,\lambda)
=
\frac 1M
\sum_{\abs{m}<M}
\p{1-\frac {\abs{m}}M}
\\
&\times
\sum_{n<N}
\e\Biggl(
  \frac 12\sum_{\ell\in\dZ}
  b^r_\ell
  \biggl(
      s_{\lambda-\mu}\p{\floor{\frac{(n+\ell)d+j}{2^\mu}}+md}
\\&&
\mathmakebox[0pt][r]{
    -
      s_{\lambda-\mu}\p{\floor{\frac{(n+\ell)d+j}{2^\mu} } } 
  \biggr)
\Biggr)
.
}
\end{align*}
The replacement of $s_\lambda$ by the two-fold restricted sum-of-digits function
$s_{\mu,\lambda}$, which we performed in~\eqref{eqn:S3_summand_estimate}, is admissible since the arguments differ by a multiple of $2^\mu$ and therefore the difference does not depend on the lower digits.
We obtain
\begin{equation}\label{eqn:S3_estimate}
S_3
\ll
N
\sum_{2^\nu\leq d<2^{\nu+1}}
\max_{j\geq 0}
\abs{S_4}
+
2^{\mu+\nu}MN
.
\end{equation}
The rough idea at this point is to estimate $S_4$ by a nonnegative real number independent of $j$, which will allow us to remove the maximum over $j$ and the absolute value appearing in~\eqref{eqn:S3_estimate}. In the following we will work out the details of this process.
We want to split the summation over $N$ into $T$ parts,
according to the fractional part of $(nd+j)2^{-\mu}$.
Let $t,T$ be integers such that $0\leq t<T$.
We define
\begin{align*}
S_5 &= S_5(N,T,a,d,j,r,m,t,\mu,\lambda)
\\&=
\sum_{
  \substack{
     n<N
  \\
      \frac tT
    \leq
      \left\{
        \frac{nd+j}{2^\mu}
      \right\}
    <
      \frac{t+1}T
  }
}
\e\Biggl(
  \frac 12\sum_{\ell\in\dZ}
  b^r_\ell
  \biggl(
      s_{\lambda-\mu}\p{\floor{\frac{(n+\ell)d+j}{2^\mu} }+md}
\\&&\mathmakebox[0pt][r]{
    -
      s_{\lambda-\mu}\p{\floor{\frac{(n+\ell)d+j}{2^\mu} } }
  \biggr)
\Biggr)
}
.
\end{align*}
We will see that, for most values of $d$,
the values of the floor function distribute evenly modulo $2^{\lambda-\mu}$
as $n$ runs through the set defined by the two conditions under the summation sign.
For this to hold, we have to assure that $N>2^{\lambda-\mu}$.
Inspecting the error terms in~\eqref{eqn:S1_estimate} and~\eqref{eqn:S3_estimate},
we see that we also need $2^\mu<N$ and $\nu<\lambda$ in order to get a nontrivial estimate.
These observations ultimately lead to the restriction $\rho_2<2$ in \propref{prp:1}.

The idea behind the decomposition into $T$ subintervals 
$[t/T,(t+1)/T)$                                         
of                                                      
$[0,1)$                                                 
is the following.
Let $A_t$ be the set of $n$ such that $\{(nd+j)/2^\mu\}$ lies in the $t$-th interval.
Then the differences
\[
\floor{\frac{(n+1)d+j}{2^\mu}}-
\floor{\frac{nd+j}{2^\mu}}
,\ldots,
\floor{\frac{(n+L-1)d+j}{2^\mu}}-
\floor{\frac{nd+j}{2^\mu}}
\]
should not depend on $n\in A_t$.
This will in fact be the case for most, but not for all, $t$, so that we have to take out some $t$.
We define a set of ``good'' values,
\begin{multline*}
G = G(T,d,r,\mu)
=
\bigg\{t<T:
  \left[
    \frac tT+\frac {\ell d}{2^\mu},\frac{t+1}T+\frac{\ell d}{2^\mu}
  \right)
\cap
  \dZ
=
  \emptyset
\\
  \text{
    for all
  }
 \ell\in [0,L)\cup[r2^\tau,r2^\tau+L)
\bigg\}.
\end{multline*}
We have
\begin{equation}\label{eqn:exceptional_intervals}
\abs{G}
\geq
T-2L
,
\end{equation}
since the intervals in the definition of $G(d,T,R,\mu)$ are disjoint and cover an interval of length $1$,
therefore we have to exclude at most one integer $t$ for each $\ell$.

We differentiate between the cases $t\in G$ and $t\not\in G$.
For $t\not\in G$ we estimate the sum in $S_5$ trivially, that is, we count the number of summands, using \lemref{lem:f_discrepancy}.
We apply this lemma for $K=2^{\lambda-\mu}$
(note that we could also take $K=1$, however, our choice spares us the separate treatment of an error term)
and multiply with $2^{\lambda-\mu}$,
which accounts for the $2^{\lambda-\mu}$ residue classes we have to collect.
We obtain
\begin{equation}\label{eqn:S5_trivial_estimate}
S_5
\ll
\frac NT
+2^{\lambda-\mu}ND_N\p{\frac {d}{2^\lambda} }
.
\end{equation}
Let $t\in G$ and assume that
$t/T\leq \{(nd+j)2^{-\mu}\}<(t+1)/T$.
By the second assumption we obtain
\[
\floor{\frac{nd+j}{2^\mu}}+\frac tT+\frac{\ell d}{2^\mu}
\leq
\frac{(n+\ell)d+j}{2^\mu}
<
\floor{\frac{nd+j}{2^\mu}}+\frac{t+1}T+\frac{\ell d}{2^\mu}
\]
and using the first assumption yields
\[
\floor{\frac{(n+\ell)d+j}{2^\mu}}
=
\floor{\frac{nd+j}{2^\mu}}+\floor{\frac tT+\frac{\ell d}{2^\mu}}
\]
for all 
$\ell\in [0,L)\cup[r2^\tau,r2^\tau+L)$. 
For $t\in G$ we obtain therefore
\begin{multline*}
S_5
=
\sum_{k<2^{\lambda-\mu}}
\sum_{
  \substack{
    n<N
  \\
      \frac tT
    \leq
      \left\{
        \frac{nd+j}{2^\mu}
      \right\}
    <
      \frac{t+1}T
  \\
    \floor{\frac{nd+j}{2^\mu}}\equiv k\bmod 2^{\lambda-\mu}
  }
}
\e\Biggl(
  \frac 12\sum_{\ell\in\dZ}
  b^r_\ell
  \biggl(
      s_{\lambda-\mu}\p{
        k
        +
        \floor{\frac tT+\frac{\ell d}{2^\mu}}
        +
        md
      }
\\
    -
      s_{\lambda-\mu}\p{
        k
        +
        \floor{\frac tT+\frac{\ell d}{2^\mu}}
      }
  \biggr)
\Biggr)
\end{multline*}
Since the summand does not depend on $n$,
we count the number of times the three conditions under the second summation sign
are satisfied.
To this end, we use again \lemref{lem:f_discrepancy} with $K=2^{\lambda-\mu}$.
We obtain for $t\in G$
\begin{multline}\label{eqn:S5_relevant_estimate}
S_5
=
\frac{N}{2^{\lambda-\mu}T}
\sum_{k<2^{\lambda-\mu}}
\e\Biggl(
  \frac 12\sum_{\ell\in\dZ}
  b^r_\ell
  \biggl(
      s_{\lambda-\mu}\p{
        k+\floor{\frac tT+\frac{\ell d}{2^\mu}}+md
      }
\\
    -
      s_{\lambda-\mu}\p{
        k+\floor{\frac tT+\frac{\ell d}{2^\mu}}
      }
  \biggr)
\Biggr)
+O\p{2^{\lambda-\mu}ND_N\p{\frac{d}{2^\lambda} } }
.
\end{multline}
Note that this expression is independent of the shift $j$.
Moreover, as we noted earlier, we see that it is necessary that we have $N\geq 2^{\lambda-\mu}$ in order to get a useful result,
since the error term would be too large otherwise.
Setting
\[    i^{d,t}_\ell  =  \floor{ \frac tT+\ell\left\{\frac d{2^\mu}\right\} },    \]
we get the almost trivial identity
\[    \floor{\frac tT+\frac{\ell d}{2^\mu}}  =  \ell\floor{\frac d{2^\mu}} + i^{d,t}_\ell.    \]
In \lemref{lem:correlation_fourier} we set $t=md$ and
\[  f(n)=\e\p{\frac 12\sum_{\ell\in\dZ}b^r_\ell s_{\lambda-\mu}\p{n+\ell\floor{\frac d{2^\mu}}+i^{d,t}_\ell}}  \]
and obtain for $t\in G$
\begin{equation}\label{eqn:S5_representation}
S_5
=
  \frac NT
  \sum_{h<2^{\lambda-\mu}}
  \abs{G_{\lambda-\mu}^{i^{d,t},b^r}\p{h,\floor{\frac d{2^\mu} } } }^2
  \e\p{\frac{hmd}{2^{\lambda-\mu} } }
+
  O\p{2^{\lambda-\mu}ND_N\p{\frac d{2^\lambda} } }
,
\end{equation}
where the Fourier coefficients $G(h,d)$ are defined by~\eqref{eqn:G_definition}.
By the definitions of $S_4$ and $S_5$ we have
\begin{equation}\label{eqn:S4_representation}
S_{4}
=
\frac 1M
\sum_{\abs{m}<M}
\p{1-\frac{\abs{m}}M}
\sum_{t<T}
S_5
.
\end{equation}
Using~\eqref{eqn:exceptional_intervals}, \eqref{eqn:S5_trivial_estimate}, \eqref{eqn:S5_representation} and~\eqref{eqn:S4_representation} we obtain
\begin{multline}\label{eqn:S4_estimate}
S_{4}
=
\frac {N}{T}
\sum_{t<T}
\frac 1M
\sum_{\abs{m}<M}
\p{1-\frac{\abs{m}}M}
\sum_{h<2^{\lambda-\mu}}
\abs{G_{\lambda-\mu}^{i^{d,t},b^r}\left(h,\floor{\frac d{2^\mu}}\right) }^2
\e\p{\frac{hmd}{2^{\lambda-\mu} } }
\\
+
O\p{
    N\frac LT
  +
    2^{\lambda-\mu}
    TND_N\p{\frac{d}{2^\lambda} }
}
,
\end{multline}
where the reinsertion of the indices $t\not\in G$ is accounted for by the error term $NL/T$, which can be seen using Parseval's identity.

Note the important fact that the right hand side gives an estimate for $S_4$ that is independent of the shift $j$
(that is, independent of the residue class modulo $d$).
Using also the nonnegativity of the main term,
which follows from the elementary identity
\[    \sum_{\abs{m}<M}(M-\abs{m})\e(mx)=\abs{\sum_{m<M}\e(mx)}^2,    \]
we may remove the maximum together with the absolute value in~\eqref{eqn:S3_estimate}
while keeping the important factor $\e\p{hmd/2^{\lambda-\mu}}$
in~\eqref{eqn:S4_estimate}.
We obtain, treating the summand $m=0$ separately,
\begin{multline}\label{eqn:S3_estimate_2}
S_3
\ll
\frac {N^2}{T}
\sum_{t<T}
\frac 1M
\sum_{1\leq \abs{m}<M}
\p{1-\frac{\abs{m}}M}
S_6
\\+
2^\nu N^2
O\Biggl(
    \frac 1M
  +
    \frac{2^{\mu}M}N
  +
    \frac LT
  +
    2^{\lambda-\mu}T
    \frac 1{2^\nu}
    \sum_{d<2^{\nu+1}}
    D_N\p{\frac d{2^\lambda} }
\Biggr)
,
\end{multline}
\vspace{-1em}where
\begin{multline}\label{eqn:S6_definition}
S_6 = S_6(d,r,m,t,\lambda,\mu,\nu)
\\=
\sum_{2^\nu\leq d<2^{\nu+1}}
\sum_{h<2^{\lambda-\mu}}
\abs{G_{\lambda-\mu}^{i^{d,t},b^r}\left(h,\floor{\frac d{2^\mu}}\right) }^2
\e\p{\frac{hmd}{2^{\lambda-\mu} } }
.
\end{multline}
We want to show that $S_6$ is substantially smaller than $2^\nu$ (which is a trivial upper bound by Parseval's identity).
In order to estimate the right hand side of~\eqref{eqn:S6_definition}, we note that the first factor depends only in a weak way on $d$.
We state this more precisely in the following.

The term $\floor{d/2^\mu}$ does not depend on the lowest $\mu$ binary digits of $d$.
Moreover, we want to decompose 
$[0,2^\mu)$                    
into few intervals $I^{r,t}_u$ in such a way that the values
$i^{d,t}_\ell$, where $\ell\in M$ and 
$M=[0,L)\cup [r2^\tau,r2^\tau+L)$,    
are constant for $d\in I^{r,t}_u$.
(Note that the indices $\ell\not\in M$ are not of interest, since $b^r_\ell=0$ for these.)

Let $t<T$ be given.
We define the (lexicographical) order on $\dN^M$ by
$(i_\ell)_{\ell\in M}<(j_\ell)_{\ell\in M}$ if and only if
$i_\ell\neq j_\ell$ for some $\ell\in M$ and
$i_\ell<j_\ell$ for $\ell=\min\{k\in M:i_k\neq j_k\}$.
It is easy to check that the assignment
\[d\mapsto i^{d,t}\mid_M\]
is $2^\mu$-periodic and nondecreasing for $d<2^\mu$ with respect to this total ordering.
It follows that the set $\{0,\ldots,2^\mu-1\}$ decomposes into intervals
$I^{r,t}_0,\ldots,I^{r,t}_{U-1}$ such that the same sequence
$\bigl(i^{d,t}_\ell\bigr)_{\ell\in M}$ is defined for each $d\in I^{r,t}_u+2^\mu\dZ$.
By the property $0\leq i^{r,t}_{\ell+1}-i^{r,t}_\ell\leq 1$,
the number $U$ of intervals thus defined satisfies
\begin{equation}\label{eqn:U_estimate}
U\leq 2^{2L}(r2^\tau-L)\leq r2^{2L+\tau}.
\end{equation}
From~\eqref{eqn:S6_definition}, we obtain sums of the form
\[    \sum_{d\in I^{r,t}_u+2^\mu k}\e\p{hmd/2^{\lambda-\mu}}    \]
for $u<U$ and some $k\in\dZ$.
Using also the estimate for the Fourier coefficients $G_{\lambda-\mu}$ from \propref{prp:uniform_fourier}, we will estimate $S_6$ nontrivially.

Let $\mathcal I=\mathcal I_{L+r2^\tau-1}$ be the set of sequences $i_0,\ldots,i_{L+r2^\tau-1}$ such that $i_0=0$ and $0\leq i_{\ell+1}-i_\ell\leq 1$ for $0\leq \ell<L+r2^\tau-1$.
Assume that $m\neq 0$ and $\lambda\geq \nu\geq \mu$.
Writing $d=d_1+2^\mu d_2$, and choosing $i_u$ such that $i_u=i^{d_1,t}$ for all $d_1\in I_u^{r,t}$, we obtain
\begin{align*}
\bigl\lvert S_6\bigr\rvert
&=
\abs{
\sum_{u<U}
\sum_{d_1\in I^{r,t}_u}
\sum_{d_2<2^{\nu-\mu}}
\sum_{h<2^{\lambda-\mu}}
\abs{G_{\lambda-\mu}^{i^{d_1,t},b^r}\p{h,d_2} }^2
\e\p{\frac{hm(d_1+d_2 2^\mu)}{2^{\lambda-\mu} } }
}
\\&=
\abs{
\sum_{u<U}
\sum_{d_2<2^{\nu-\mu}}
\sum_{h<2^{\lambda-\mu}}
\abs{G_{\lambda-\mu}^{i_u,b^r}\p{h,d_2} }^2
\sum_{d_1\in I^{r,t}_u}
\e\p{\frac{hm(d_1+d_2 2^\mu)}{2^{\lambda-\mu} } }
}
\\&\leq
\sum_{u<U}
\max_{i\in \mathcal I}
\sum_{d_2,h<2^{\lambda-\mu}}
\abs{G_{\lambda-\mu}^{i,b^r}\p{h,d_2} }^2
\abs{
  \sum_{d_1\in I^{r,t}_u}
  \e\p{\frac{hmd_1}{2^{\lambda-\mu} } }
}.
\end{align*}
We apply the Cauchy--Schwarz inequality to the sum over $h$ and $d_2$ and obtain with the help of Parseval's identity
\begin{align}
\bigl\lvert S_6\bigr\rvert
&\leq
\sum_{u<U}
\max_{i\in \mathcal I}
\p{
  \sum_{d_2,h<2^{\lambda-\mu}}
  \abs{G_{\lambda-\mu}^{i,b^r}(h,d_2)}^4
}^{1/2}\nonumber
\\&&
\mathmakebox[0pt][r]{
  \times
  \p{
    \sum_{d_2,h<2^{\lambda-\mu}}
    \abs{
      \sum_{d_1\in I^{r,t}_u}
      \e\p{\frac{hmd_1}{2^{\lambda-\mu} } }
    }^2
  }^{1/2}
}
\nonumber
\\&\leq
2^{(\lambda-\mu)/2}
\max_{i\in \mathcal I}
\p{
  \sum_{d_2<2^{\lambda-\mu}}
  \max_{h<2^{\lambda-\mu}}
  \abs{G_{\lambda-\mu}^{i,b^r}(h,d_2)}^2
}^{1/2}
\label{eqn:showdown_1}
\\&&
\mathmakebox[0pt][r]{
  \times
  \sum_{u<U}
  \p{
    \sum_{h<2^{\lambda-\mu}}
    \abs{
      \sum_{k<2^{\lambda-\mu}}
      a_k^u
      \e\p{\frac{hk}{2^{\lambda-\mu} } }
    }^2
  }^{1/2}
}\nonumber
,
\end{align}
where
\[    a_k^u=a_k^{u,r,t,m}=\abs{\{d_1\in I^{r,t}_u:d_1m\equiv k\bmod 2^{\lambda-\mu}\}}.    \]
In order to estimate the sum over $d_2$ in~\eqref{eqn:showdown_1}, we use \propref{prp:uniform_fourier}:
there exist $\eta>0$ and $C$, which only depend on $L$,
such that for all $\lambda\geq\mu\geq 0$ and all $r\geq 1$ satisfying
$\nu_2(r)+\tau\leq (\lambda-\mu)/4$ we have
\begin{equation}\label{eqn:fourth_powers_estimate}
\max_{i\in \mathcal I}
\p{
  \sum_{d_2<2^{\lambda-\mu}}
  \max_{h<2^{\lambda-\mu}}
  \abs{G_{\lambda-\mu}^{i,b^r}(h,d_2)}^2
}^{1/2}
\leq
C
2^{(1-\eta)(\lambda-\mu)/2}.
\end{equation}
The sum over $h$ in~\eqref{eqn:showdown_1} can be estimated by \lemref{lem:large_sieve} (the large sieve inequality) and the estimate
\[
a_k^{u,r,t,m}
\leq
\begin{cases}
2^{\nu_2(m)}&\mu\leq \lambda-\mu\\
2^{\nu_2(m)}2^\mu/2^{\lambda-\mu}&\mu>\lambda-\mu
\end{cases}
\]
(which is clear for odd $m$ and follows easily by the decomposition $m=m_02^{\nu_2(m)}$ otherwise).
This gives
\begin{equation}\label{eqn:sieve_estimate}
\begin{aligned}
\p{
  \sum_{h<2^{\lambda-\mu}}
  \abs{
    \sum_{k<2^{\lambda-\mu}}
    a_k^{u,r,t,m}
    \e\p{\frac{hk}{2^{\lambda-\mu} } }
  }^2
}^{1/2}
&\ll
\p{
  2^{\lambda-\mu}
  \sum_{k<2^{\lambda-\mu}}
  \bigl\lvert a_k^{u,r,t,m}\bigr\rvert^2
}^{1/2}
\\&\ll
  2^{\lambda-\mu}
  2^{\nu_2(m)}
  \max\{1,2^{2\mu-\lambda}\}
.
\end{aligned}
\end{equation}
Combining~\eqref{eqn:fourth_powers_estimate}, \eqref{eqn:sieve_estimate} and~\eqref{eqn:U_estimate}, we get
\begin{align*}
\abs{S_6}
&\ll
r2^{2L+\tau}2^{\nu_2(m)}\max\{1,2^{2\mu-\lambda}\}2^{(2-\eta/2)(\lambda-\mu)}
\\&\ll
r2^{\nu_2(m)}\max\{1,2^{2\mu-\lambda}\}2^{(2-\eta/2)(\lambda-\mu)}
.
\end{align*}
with an implied constant depending only on $L$.
We translate this estimate back to an estimate for $S_1$.
By the estimate
\[    \sum_{1\leq m<M}2^{\nu_2(m)}\ll M\log M    \]
valid for $M\geq 1$, which is easy to show by splitting the summation according to the value of $\nu_2(m)$, we obtain
\begin{equation}\label{eqn:sum_S6_estimate}
\frac 1{RM}
\sum_{1\leq r<R}
\sum_{1\leq \abs{m}<M}\p{1-\frac{\abs{m}}M}S_6
\ll
R\log M\max\{1,2^{2\mu-\lambda}\}2^{(2-\eta/2)(\lambda-\mu)}
.
\end{equation}
We collect the error terms, using~\eqref{eqn:S1_estimate},
\eqref{eqn:S3_estimate_2} and~\eqref{eqn:sum_S6_estimate}
and use the discrepancy estimate~\eqref{eqn:mean_discrepancy_sum}, obtaining
\begin{multline}\label{eqn:collected_errors}
\abs{\frac{S_1(N,2^\nu,\xi)}{N2^\nu}}^4
\leq
C
\biggl(
    \frac 1{R^2}
  +
    \p{\frac{R2^\nu}{2^\lambda}}^2
  +
    \p{\frac RN}^2
  +
    \frac 1M
  +
    \frac{2^\mu M}N
  +
    \frac 1T
\\
  +
    T
    \frac{2^{\lambda-\mu}}N
    \frac{N+2^\lambda}{2^\nu}
    \bigl(\logp N\bigr)^2
  +
    R
    2^{(2-\eta/2)(\lambda-\mu)-\nu}
    \log M
    \max\{1,2^{2\mu-\lambda}\}
\biggr)
\end{multline}
with a constant $C$ depending only on $L$.
This estimate is valid for all integers $M,N,R,T\geq 1$ and $\lambda,\mu,\nu\geq 0$ such that $\mu\leq\nu<\nu+1\leq \lambda$ and $R2^\tau\leq 2^{(\lambda-\mu)/4}$, and for all real numbers $\xi$.
Moreover, this estimate also holds for real-valued parameters $M,R,T,\lambda,\mu$ satisfying these restrictions (with a possibly different constant $C$).
In order to finish the proof of \propref{prp:1}, we have to choose the parameters $M,R,T,\lambda$ and $\mu$, depending on $N$ and $D$.
Let $0<\rho_1\leq \rho_2<2$ be given and choose $\theta$ and $\varepsilon$ in such a way that
\[
\max\biggl\{
1,
\rho_2,
\frac{3\rho_2}{1+\rho_2},
2-\eta/2\biggr
\}<\theta<2
\]
and
\[
0<\varepsilon<\min\biggl\{
2-\theta,
\theta/\rho_2-1,
\theta-1,
\theta\frac{1+\rho_2}{\rho_2}-3,
\theta-(2-\eta/2),
1/4
\biggr\}.
\]
Assume that $D\geq 1$ is a real number and that $N,\nu\geq 1$ are integers such that
$N^{\rho_1}\leq D\leq N^{\rho_2}$ and $D<2^\nu\leq 2D$.
Set $\mu=\nu/\theta$, $\lambda=2\mu$ and $R=M=T=2^{\varepsilon\mu}$.
Using these choices it is not difficult to check, proceeding term by term, that
\[    \frac{S_1(N,2^\nu,\xi)}{N2^\nu}\leq C 2^{-\nu\eta_1}\bigl(\logp N\bigr)^2    \]
for some $C$ and $\eta_1>0$ depending only on $\rho_2$ and $L$.
Finally, we insert the lower bound $N^{\rho_1}\leq 2^\nu$ (so far we did not use $\rho_1$),
which completes the proof of \propref{prp:1}.
\section{Proof of \propref{prp:2}}\label{sec:proof_prp2}
We follow the proof of \propref{prp:1} and start, without loss of generality, with the same assumptions.
Assume that $L\geq 1$ is an integer, $a_0=1$ and $a_1,\ldots,a_{L-1}\in\{0,1\}$.
Choose $m\geq 5$ such that $2^{m-5}\leq L<2^{m-4}$ and set $\tau=2m$.
Assume that $D,N,\nu\geq 1$, where $N$ and $\nu$ are integers
satisfying $N^{\rho_1}\leq N\leq N^{\rho_2}$ and $D<2^\nu\leq 2D$.

We apply van der Corput's inequality for $K=2^\tau$ and obtain, in analogy to~\eqref{eqn:S1_estimate},
\begin{multline}\label{eqn:tilde_S1_estimate}
\abs{\tilde S_1}
^2
\ll
(2^\nu N)^2
O\p{
    \frac 1R
  +
    \frac{R2^\tau 2^\nu}{2^\lambda}
  +
    \frac {R2^\tau}N
}
\\
+
2^{3\nu/2}N
\p{
\frac 1R
\sum_{1\leq r<R}
\tilde S_3(N,r,\lambda,\nu)
}^{1/2}
,
\end{multline}
where
\begin{equation}\label{eqn:tilde_S3_definition}
\tilde S_3(N,r,\lambda,\nu)
=
\int_{2^\nu}^{2^{\nu+1}}
\max_{\beta\geq 0}
\abs{
  \sum_{n<N}
  \e\p{
    \frac 12
    \sum_{\ell\in\dZ}
    b^r_\ell
    s_\lambda\p{\floor{(n+\ell)\alpha+\beta}}
  }
}^2
\,\mathrm d \alpha
\end{equation}
and $b^r_\ell=a_{\ell-r2^\tau}-a_\ell$, and $a_\ell$ is assumed to be zero for 
$\ell\not\in[0,L)$. 
The estimate~\eqref{eqn:tilde_S1_estimate} is valid for $N,R\geq 1$ and $\lambda,\nu\geq 0$.

Next we want to ``cut off'' the $\mu$ lowest digits, that is, replace $s_\lambda$ by $s_{\mu,\lambda}$. This was accomplished by a simple application of \lemref{lem:vdc}, setting $K=2^\mu$, in the proof of \propref{prp:1}.
Since we are now considering Beatty sequences (having in general non-integer step length $\alpha$), we need to modify our strategy.
To do so, we use Diophantine approximation, more precisely, Farey series.
Let $\alpha\in\dR$ be given.
We assign a fraction $p(\alpha)/q(\alpha)$ to $\alpha$ according to the Farey dissection of the circle:
consider reduced fractions $a/b<c/d$ that are neighbours in the Farey series $\mathcal F_{2^{\mu+\sigma}}$,
where $\sigma\geq 1$ is chosen later, such that $a/b\leq \alpha/2^\mu<c/d$.
If $\alpha/2^\mu<(a+c)/(b+d)$, then set $p(\alpha)=a$ and $q(\alpha)=b$, otherwise set $p(\alpha)=c$ and $q(\alpha)=d$.
\lemref{lem:farey} implies
\begin{equation}\label{eqn:dirichlet}
\bigl\lvert q(\alpha)\alpha-p(\alpha)2^\mu\bigr\rvert < 2^{-\sigma}.
\end{equation}
Applying \lemref{lem:vdc} with $K=q(\alpha)$ and noting that $q(\alpha)\leq 2^{\mu+\sigma}$,
we obtain for all integers $M\geq 1$ and $\mu\geq 0$
\begin{multline}
\label{eqn:tilde_S3_summand_estimate}
\abs{
  \sum_{n<N}
  \e\p{
    \frac 12
    \sum_{\ell\in\dZ}
    b^r_\ell
    s_\lambda\p{\floor{(n+\ell)\alpha+\beta}}
  }
}^2
\\
\ll
O(2^{\mu+\sigma}MN)
+
\frac NM
\sum_{\abs{m}<M}
\p{1-\frac{\abs{m}}M}
\\\times
\sum_{n<N}
\e\p{
  \frac 12
  \sum_{\ell\in\dZ}
  b^r_\ell
  \bigl(
      s_{\lambda}\p{\floor{(n+\ell+m\,q(\alpha))\alpha+\beta} }
    -
      s_{\lambda}\p{\floor{(n+\ell)\alpha+\beta} }
  \bigr)
}.
\end{multline}
In order to reduce this expression to a sum analogous to $S_4$, we want to shift the expression $m\,q(\alpha)\alpha$ out of the floor function.
To this end, we use~\eqref{eqn:dirichlet} and the argument that $m\,q(\alpha)\alpha$ is close to an integer, while $(n+\ell)\alpha+\beta$ usually is not.
This can be made precise as follows.
Assume that
\begin{equation}\label{eqn:spaced}
\norm{(n+\ell)\alpha+\beta}\geq M/2^\sigma
\end{equation}
and that $2M<2^\sigma$.
Using part~\eqref{item:fpf_3} of \lemref{lem:fractional_part_facts} with $\varepsilon=1/2^\sigma$
and~\eqref{eqn:dirichlet}, moreover noting that $\sigma\geq 1$, we obtain
\[    \ang{m\,q(\alpha)\alpha}=m\ang{q(\alpha)\alpha}=m\,p(\alpha)2^\mu.    \]
Applying part~\eqref{item:fpf_1} of \lemref{lem:fractional_part_facts}, setting $\varepsilon=M/2^\sigma$, we see that~\eqref{eqn:spaced} implies
\begin{equation}\label{eqn:floor_split}
\floor{(n+\ell+m\,q(\alpha))\alpha+\beta}
=
m\,p(\alpha)2^\mu+\floor{(n+\ell)\alpha+\beta}.
\end{equation}
The number of $n$ where hypothesis~\eqref{eqn:spaced} fails for some $\ell$ can be estimated by discrepancy estimates for $\{n\alpha\}$-sequences:
for all positive integers $N$ and $2M<2^\sigma$ we have
\begin{equation}\label{eqn:count_close_integers}
\begin{aligned}
\hspace{4em}&\hspace{-4em}\bigl\lvert\bigl\{n<N:\norm{(n+\ell)\alpha+\beta}\leq M/2^\sigma\bigr\}\bigr\rvert\\
&=\bigl\lvert\bigl\{n<N:(n+\ell)\alpha+\beta\in \left[-M/2^\sigma,M/2^\sigma\right]+\dZ\bigr\}\bigr\rvert\\
&=\bigl\lvert\bigl\{n<N:n\alpha\in \left[0,2M/2^\sigma\right]-\beta-\ell\alpha-M/2^\sigma+\dZ\bigr\}\bigr\rvert\\
&\leq ND_N(\alpha)+2MN/2^\sigma.
\end{aligned}
\end{equation}
Multiplying this error by $2L$ (which is $O(1)$ according to our conventions stated in the introduction),
we obtain an upper bound for the number of $n<N$ such that
$\norm{(n+\ell)\alpha+\beta}\leq M/2^\sigma$ for some 
$\ell\in[0,L)\cup[r2^\tau,r2^\tau+L)$.                
Treating these integers separately and using~\eqref{eqn:tilde_S3_summand_estimate} through~\eqref{eqn:count_close_integers}, we obtain
\begin{multline}\label{eqn:tilde_S3_summand_estimate_2}
\abs{
  \sum_{n<N}
  \e\p{
    \frac 12
    \sum_{\ell\in\dZ}
    b^r_\ell
     s_\lambda\p{\floor{(n+\ell)\alpha+\beta}}
  }
}^2
\\
\ll
N\bigl\lvert\tilde S_4\bigr\rvert
+
O\bigl(2^{\mu+\sigma}NM
+
N^2 D_N(\alpha)
+
N^2M/2^\sigma
\bigr)
,
\end{multline}
where
\begin{align*}
\tilde S_4 &= \tilde S_4(N,M,\alpha,\beta,r,\mu,\lambda)
=
\frac 1M
\sum_{\abs{m}<M}
\p{1-\frac {\abs{m}}M}
\\&\times
\sum_{n<N}
\e\Biggl(
 \frac 12
 \sum_{\ell\in\dZ}
 b^r_\ell
 \biggl(
  s_{\lambda-\mu}\p{\floor{\frac{(n+\ell)\alpha+\beta}{2^\mu} }+m\,p(\alpha)}
\\&&
\mathmakebox[0pt][r]{
   -s_{\lambda-\mu}\p{\floor{\frac{(n+\ell)\alpha+\beta}{2^\mu} } }
  \biggr)
 \Biggr)
.
}
\end{align*}
Note that~\eqref{eqn:tilde_S3_summand_estimate_2} is, except for the error terms,
completely analogous to equation~\eqref{eqn:S3_summand_estimate} in the proof of \propref{prp:1}.
From~\eqref{eqn:tilde_S3_definition} and~\eqref{eqn:tilde_S3_summand_estimate_2} we get
\begin{equation}\label{eqn:tilde_S3_estimate}
\begin{aligned}
\tilde S_3(N,r,\lambda,\nu)
&\ll
N
\int_{2^\nu}^{2^{\nu+1}}
\max_{\beta\geq 0}
\bigl\lvert\tilde S_4\bigr\rvert
\,\mathrm d\alpha
+
N^2\int_{2^\nu}^{2^{\nu+1}}D_N(\alpha)\,\mathrm d\alpha
\\&+
2^{\mu+\sigma+\nu}MN
+
N^22^\nu M/2^\sigma
.
\end{aligned}
\end{equation}
Let $t,T$ be integers such that $0\leq t<T$ and define
\begin{multline*}
\tilde S_5 
=
\sum_{
  \substack{
     n<N
  \\
      \frac tT
    \leq
      \left\{
        \frac{n\alpha+\beta}{2^\mu}
      \right\}
    <
      \frac{t+1}T
  }
}
\e\Bigg(
  \frac 12\sum_{\ell\in\dZ}
  b^r_\ell
  \biggl(
      s_{\lambda-\mu}\p{\floor{\frac{(n+\ell)\alpha+\beta}{2^\mu} }+m\,p(\alpha)}
\\
    -
      s_{\lambda-\mu}\p{\floor{\frac{(n+\ell)\alpha+\beta}{2^\mu} } }
  \biggr)
\Bigg)
\end{multline*}
and
\begin{multline*}
G=G(T,\alpha,r,\mu)
=
\biggl\{t<T:
  \left[
    \frac tT+\frac {\ell \alpha}{2^\mu},
    \frac{t+1}T+\frac{\ell \alpha}{2^\mu}
  \right)
\cap
  \dZ
=
  \emptyset
\\
  \text{ for all }
  \ell\in [0,L)\cup[r2^\tau,r2^\tau+L)
\biggr\}.
\end{multline*}
Again, we have
\begin{equation}\label{eqn:exceptional_intervals2}
G
\geq
T-2L
\end{equation}
and we distinguish between the cases $t\in G$ and $t\not\in G$.
For $t\not\in G$ we estimate $\tilde S_5$ trivially, applying \lemref{lem:f_discrepancy} with $K=2^{\lambda-\mu}$.
We obtain
\begin{equation}\label{eqn:tilde_S5_trivial_estimate}
\tilde S_5
\ll
\frac NT
+
2^{\lambda-\mu}ND_N\p{\frac {\alpha}{2^\lambda} }
.
\end{equation}
Let $t\in G$ 
and assume that $t/T\leq \{(n\alpha+\beta)2^{-\mu}\}<(t+1)/T$.
Then, as before,
\[
\floor{\frac{(n+\ell)\alpha+\beta}{2^\mu}}
=
\floor{\frac{n\alpha+\beta}{2^\mu}}+\floor{\frac tT+\frac{\ell \alpha}{2^\mu}}
\]
for $\ell\in [0,L)\cup[r2^\tau,r2^\tau+L)$. 
For $t\in G$ we obtain
\begin{multline*}
\tilde S_5
=
\sum_{k<2^{\lambda-\mu}}
\sum_{
  \substack{
    n<N
  \\
      \frac tT
    \leq
      \left\{
        \frac{n\alpha+\beta}{2^\mu}
      \right\}
    <
      \frac{t+1}T
  \\
    \floor{\frac{n\alpha+\beta}{2^\mu}}\equiv k\bmod 2^{\lambda-\mu}
  }
}
\e\Biggl(
  \frac 12\sum_{\ell\in\dZ}
  b^r_\ell
  \biggl(
      s_{\lambda-\mu}
      \p{
        k
        +
        \floor{\frac tT+\frac{\ell \alpha}{2^\mu}}
        +
        m\,p(\alpha)
      }
\\
    -
      s_{\lambda-\mu}\p{
        k
        +
        \floor{\frac tT+\frac{\ell \alpha}{2^\mu}}
      }
  \biggr)
\Biggr)
.
\end{multline*}
We apply \lemref{lem:f_discrepancy}, setting $K=2^{\lambda-\mu}$, and obtain for $t\in G$
\begin{multline}\label{eqn:tilde_S5_relevant_estimate}
\tilde S_5
=
\frac{N}{2^{\lambda-\mu}T}
\sum_{k<2^{\lambda-\mu}}
\e\Biggl(
  \frac 12\sum_{\ell\in\dZ}
  b^r_\ell
  \biggl(
      s_{\lambda-\mu}
      \p{
        k+\floor{\frac tT+\frac{\ell \alpha}{2^\mu}}+m\,p(\alpha)
      }
\\
    -
      s_{\lambda-\mu}\p{
        k+\floor{\frac tT+\frac{\ell \alpha}{2^\mu}}
      }
  \biggl)
\Biggr)
+O\p{2^{\lambda-\mu}ND_N\p{\frac \alpha{2^\lambda} } }.
\end{multline}
Setting
\[    i^{\alpha,t}_\ell=\floor{\frac tT+\ell\left\{\frac\alpha{2^\mu}\right\}},    \]
we have
\[    \floor{\frac tT+\frac{\ell\alpha}{2^\mu}}=\ell\floor{\frac \alpha{2^\mu}}+i^{\alpha,t}_\ell.    \]
In \lemref{lem:correlation_fourier} we set $t=m\,p(\alpha)$ and
\[f(n)=\e\p{\frac 12\sum_{\ell\in\dZ}b^r_\ell
s_{\lambda-\mu}\p{n+\ell\floor{\frac{\alpha}{2^\mu}}+i^{\alpha,t}_\ell} }\]
and obtain for $t\in G$
\begin{equation}\label{eqn:tilde_S5_representation}
\tilde S_5
=
  \frac NT
  \sum_{h<2^{\lambda-\mu}}
  \abs{G_{\lambda-\mu}^{i^{\alpha,t},b^r}\p{h,\floor{\frac \alpha{2^\mu} } } }^2
  \e\p{\frac{hm\,p(\alpha)}{2^{\lambda-\mu} } }
+
  O\p{2^{\lambda-\mu}ND_N\p{\frac \alpha{2^\lambda} } }
.
\end{equation}
Using the identity
\[
\tilde
S_{4}
=
\frac 1M
\sum_{\abs{m}<M}
\p{1-\frac{\abs{m}}M}
\sum_{t<T}
\tilde
S_5
\]
as well as~\eqref{eqn:exceptional_intervals2},~\eqref{eqn:tilde_S5_trivial_estimate} and~\eqref{eqn:tilde_S5_representation}, we obtain, in analogy to~\eqref{eqn:S4_estimate},
\begin{align}
\tilde S_{4}
&=
\frac {N}{T}
\sum_{t<T}
\frac 1M
\sum_{\abs{m}<M}
\p{1-\frac{\abs{m}}M}\nonumber
\\&\qquad\times
\sum_{h<2^{\lambda-\mu}}
\abs{G_{\lambda-\mu}^{i^{\alpha,t},b^r}
\left(h,\floor{\frac \alpha{2^\mu}}\right) }^2
\e\p{\frac{hm\,p(\alpha)}{2^{\lambda-\mu} } }
\label{eqn:tilde_S4_estimate}
\\&\qquad+
O\p{
    \frac NT
  +
    2^{\lambda-\mu}
    TND_N\p{\frac{\alpha}{2^\lambda} }
}\nonumber
.
\end{align}
From~\eqref{eqn:tilde_S3_estimate} and~\eqref{eqn:tilde_S4_estimate},
treating the summand for $m=0$ separately, we get
\begin{multline}\label{eqn:tilde_S3_estimate_2}
\tilde S_3
\ll
\frac {N^2}{T}
\sum_{t<T}
\frac 1M
\sum_{1\leq \abs{m}<M}
\p{1-\frac{\abs{m}}M}
\tilde S_6
+
2^\nu N^2
\,O\Biggl(
    \frac 1M
  +
    \frac{2^{\mu+\sigma}M}N
\\
  +
    \frac 1T
  +
    \frac {2^{\lambda-\mu}T}{2^\nu}
    \int_{2^\nu}^{2^{\nu+1} }
    D_N\p{\frac {\alpha}{2^\lambda}}
    \,\mathrm d\alpha
  +
    \frac 1{2^\nu}
    \int_{2^\nu}^{2^{\nu+1} }
    D_N(\alpha)
    \,\mathrm d\alpha
  +
    \frac{M}{2^\sigma}
\Biggr)
,
\end{multline}
where
\begin{multline*}
\tilde S_6
=
\tilde S_6(\alpha,r,m,t,\lambda,\mu,\nu)
\\=
\int_{2^\nu}^{2^{\nu+1} }
\sum_{h<2^{\lambda-\mu} }
\abs{G_{\lambda-\mu}^{i^{\alpha,t},b^r}
\left(h,\floor{\frac \alpha{2^\mu}}\right) }^2
\e\p{\frac{hm\,p(\alpha)}{2^{\lambda-\mu} } }
\,\mathrm d\alpha
.
\end{multline*}
As in the proof of \propref{prp:1} we obtain intervals $I^{r,t}_0,\ldots,I^{r,t}_{U-1}\subseteq\dR$, where $U\leq r2^{2L+\tau}$, such that $\alpha\mapsto i^{\alpha,t}\mid_M$ is constant on each $I^{r,t}_u$.
Define $\mathcal I$ as before, that is, as the set of sequences $(i_\ell)_{\ell<L+r2^\tau-1}$ such that $0\leq i_{\ell+1}-i_\ell\leq 1$ for $0\leq i<L+r2^\tau-1$.
Moreover, we assume from now on that $m\neq 0$ and $\lambda\geq \nu\geq \mu$.
We obtain, applying the Cauchy--Schwarz inequality to the sum over $(h,d_2)$,
\begin{align*}
\bigl\lvert\tilde S_6\bigr\rvert
&=
\abs{
\sum_{u<U}
\int_{I^{r,t}_u}
\sum_{d_2<2^{\nu-\mu}}
\sum_{h<2^{\lambda-\mu}}
\abs{G_{\lambda-\mu}^{i^{\alpha,t},b^r}(h,d_2) }^2
\e\p{\frac{hm\,p(\alpha+d_2 2^\mu)}{2^{\lambda-\mu} } }\,\mathrm d\alpha
}
\\
&\leq
\sum_{u<U}
\max_{i\in \mathcal I}
\sum_{d_2<2^{\lambda-\mu}}
\sum_{h<2^{\lambda-\mu}}
\abs{G_{\lambda-\mu}^{i,b^r}(h,d_2) }^2
\abs{
  \int_{I^{r,t}_u}
  \e\p{\frac{hm\,p(\alpha+d_2 2^\mu)}{2^{\lambda-\mu} } }
  \,\mathrm d\alpha
}
\\&\leq
\sum_{u<U}
\max_{i\in \mathcal I}
\p{
  \sum_{d_2,h<2^{\lambda-\mu}}
  \abs{G^{i,b^r}_{\lambda-\mu}(h,d_2)}^4
}^{1/2}
\\&&
\mathmakebox[0pt][r]{
  \times
  \p{
    \sum_{d_2,h<2^{\lambda-\mu}}
    \abs{
      \int_{I^{r,t}_u}
      \e\p{\frac{hm\,p(\alpha+d_2\,2^\mu)}{2^{\lambda-\mu} } }
      \,\mathrm d\alpha
    }^2
  }^{1/2}
}
\\&\leq
\max_{i\in \mathcal I}
\p{
  \sum_{d_2<2^{\lambda-\mu}}
  \max_{h<2^{\lambda-\mu}}
  \abs{G^{i,b^r}_{\lambda-\mu}(h,d_2)}^2
}^{1/2}
\\&&
\mathmakebox[0pt][r]{
  \times
  \sum_{u<U}
  \p{
    \sum_{d,h<2^{\lambda-\mu}}
    \abs{
      \sum_{k<2^{\lambda-\mu}}
      a_k^{u,d}
      \e\p{\frac{hk}{2^{\lambda-\mu}} }
    }^2
  }^{1/2}
}
,
\end{align*}
where
\[
a_k^{u,d}
=a_k^{r,m,t,u,d,\mu,\lambda}
=\boldsymbol\lambda\p{
  \{
    \alpha\in I_u^{r,t}:m\,p(\alpha+d\,2^\mu)\equiv k\bmod 2^{\lambda-\mu}
  \}
},
\]
and~$\boldsymbol\lambda$ denotes the Lebesgue measure.
Assume that $r2^\tau\leq 2^{(\lambda-\mu)/4}$.
Then, using \propref{prp:uniform_fourier},
the first factor can be estimated by $C2^{(1-\eta)(\lambda-\mu)/2}$.
The second factor can be estimated by the large sieve inequality
(\lemref{lem:large_sieve}) and $\lvert U\rvert\ll r$, which yields
\[
\bigl\lvert\tilde S_6\bigr\rvert
\ll
r\,2^{(1-\eta)(\lambda-\mu)/2}
\max_{u<U}
\p{
\sum_{d<2^{\lambda-\mu}}
2^{\lambda-\mu}
\sum_{k<2^{\lambda-\mu}}
\abs{
a_k^{u,d}
}^2
}^{1/2}
.
\]
It remains to find estimates for $a_k^{u,d}$.
In order to do so, we first note that
\begin{equation}\label{eqn:fraction_shift}
\begin{aligned}
q(\alpha+d\,2^\mu)&=q(\alpha)\\
p(\alpha+d\,2^\mu)&=p(\alpha)+q(\alpha)d
\end{aligned}
\end{equation}
for all $d\in\dN$.
\lemref{lem:farey} implies
\begin{equation}\label{eqn:measure_estimate}
\boldsymbol\lambda\p{\{\alpha\in[0,2^\mu):p(\alpha)=p, q(\alpha)=q\}}
\leq
2\frac{2^\mu}{2^{\mu+\sigma}q}
.
\end{equation}
Formulas~\eqref{eqn:measure_estimate} and~\eqref{eqn:fraction_shift} together with the estimate
\[
\abs{\{p<2^{\mu+\sigma}:mp\equiv k\bmod 2^{\lambda-\mu}\}}
\leq
2^{\nu_2(m)}\max\{1,2^{2\mu-\lambda+\sigma}\}
\]
imply
\begin{align*}
\hspace{4em}&\hspace{-4em}\boldsymbol\lambda\p{\{\alpha\in[0,2^\mu):m\,p(\alpha+d\,2^\mu)\equiv k\bmod 2^{\lambda-\mu}, q(\alpha+d\,2^\mu)=q\}}
\\&=
\boldsymbol\lambda\p{\{\alpha\in[0,2^\mu):m\,p(\alpha)\equiv k-d\,q\bmod 2^{\lambda-\mu}, q(\alpha)=q\}}
\\&\leq
2^{\nu_2(m)}
\max\{1,2^{2\mu-\lambda+\sigma}\}
\frac{2}{2^\sigma q}
\end{align*}
By a summation over $q\leq 2^{\mu+\sigma}$ we obtain
\[
a_k^{u,d}
\ll
2^{\nu_2(m)}
\max\{1,2^{2\mu-\lambda+\sigma}\}
(\mu+\sigma)/2^{\sigma}
.
\]
Therefore
\[\bigl\lvert\tilde S_6\bigr\rvert
\ll
R2^{\nu_2(m)}(\mu+\sigma)
\,2^{(2-\eta')(\lambda-\mu)}\max\{2^{-\sigma},2^{2\mu-\lambda}\}
,\]
which leads to
\begin{equation}\label{eqn:tilde_sum_S6_estimate}
\begin{aligned}
\hspace{4em}&\hspace{-4em}
\frac 1{RM}
\sum_{1\leq r<R}
\sum_{1\leq\abs{m}<M}
\p{1-\frac{\abs{m}}M}
\tilde S_6
\\&\ll
R\log M (\mu+\sigma)
\,2^{(2-\eta')(\lambda-\mu)}
\max\{2^{-\sigma},2^{2\mu-\lambda}\}.
\end{aligned}
\end{equation}
By analogous reasoning as in the proof of \propref{prp:2}, leading to~\eqref{eqn:collected_errors}, and using
\eqref{eqn:tilde_S1_estimate}, \eqref{eqn:tilde_S3_estimate_2}
and~\eqref{eqn:tilde_sum_S6_estimate} and the estimate for the integral mean of
the discrepancy (\lemref{lem:mean_discrepancy}), we obtain
\begin{multline*}
\abs{\frac{\tilde S_1(N,\nu,\xi)}{N2^\nu}}^4
\ll
\frac 1{R^2}
+
\p{\frac{R2^{\nu+\tau}}{2^\lambda } }^2
+
\p{\frac {R2^\tau}N}^2
+
\frac 1M
+
\frac{2^{\mu+\sigma}M}N
+
\frac 1T
+
\frac {M}{2^\sigma}
\\
+
T
\frac{2^{\lambda-\mu}}N
\bigl(\logp N\bigr)^2
+
R\log M(\mu+\sigma)2^{(2-\eta')(\lambda-\mu)+2\mu-\lambda-\nu}
\end{multline*}
for all integers $T,R,M,N,\sigma\geq 1$ and $\lambda,\mu,\nu\geq 0$
such that $\mu\leq \nu\leq \lambda\leq 2\nu$, $2M<2^\sigma$,
$R2^\tau\leq 2^{(\lambda-\mu)/4}$ and for all $\xi\in\dR$.
An analogous argument as in the proof of \propref{prp:1} finishes the proof.
\section{Proof of \propref{prp:uniform_fourier}}
\subsection{A recurrence for Fourier coefficients}
In order to get started with the proof of \propref{prp:uniform_fourier}, we recall the definition of the discrete Fourier coefficients $G_\lambda^{i,b}(h,d)$, which is equation~\eqref{eqn:G_definition}.
For nonnegative integers $d$ and $\lambda$, for sequences
$i:\dN\ra\dN$ (for notational reasons we use, in this section, the function notation $i(\ell)$ to denote the $\ell$-th element) and $b:\dN\ra\dZ$, where $b$ has finite support,
and for $h\in\dZ$ we have
\begin{equation*}
G_\lambda^{i,b}(h,d)
=
\frac 1{2^\lambda}
\sum_{u<2^\lambda}
\e\p{\frac 12
\sum_{\ell\geq 0}
b_\ell s_\lambda(u+\ell d+i(\ell))-\frac{hu}{2^\lambda}
}.
\end{equation*}
For any $K\geq 0$, we denote by $\mathcal I_K$ the set of sequences $i:\dN\ra\dN$ with support in 
$[0,K)$ 
such that $i(\ell+1)-i(\ell)\in \{0,1\}$ for $0\leq \ell<K-1$.
For $(\delta,\varepsilon)\in \{0,1\}^2$, we define a transformation $T_{\delta,\varepsilon}:\mathcal I_K\ra\mathcal I_K$ by
\begin{equation*}
  T_{\delta,\varepsilon}(i)(\ell) = \floor{\frac{i(\ell) + \ell \delta + \varepsilon}{2}}
\end{equation*}
for $0\leq \ell<K$.
Note that this transformation is well-defined.
We also define \emph{weights} by
\begin{equation*}
  f^{i,b}_{\delta,\varepsilon} = \e\p{\frac{1}{2}\sum_{\ell<K} b_{\ell}(i(\ell)+\ell\delta+\varepsilon)}.
\end{equation*}
These quantities appear in the following recurrence for the discrete Fourier coefficients,
compare~\cite[Lemma 13]{DMR2015}.
\begin{lem}
Assume that $b:\dN\ra\dZ$ and $i\in \mathcal I_K$.
Then for all integers $d,\lambda\geq 0$ and $h$,
and $\varepsilon \in \{0,1\}$ we have
\begin{equation}\label{eqn:rec_G}
  G_{\lambda}^{i,b}(h,2 d+\delta)
  = \frac{1}{2} \sum_{\varepsilon=0}^{1}  \e\p{-\frac{h\varepsilon}{2^{\lambda} } } f_{\delta,\varepsilon}^{i,b}\, G_{\lambda-1}^{T_{\delta,\varepsilon}(i),b}(h,d).
\end{equation}
\end{lem}
\begin{proof}
By splitting the sum in the definition of $G^{i,b}_\lambda(h,d)$ according to the parity of $u$, we obtain (writing $\varepsilon_0(n)$ for the lowest binary digit of $n$)
\begin{align*}
\hspace{1em}&\hspace{-1em}G_{\lambda}^{i, b}(h,2d + \delta)
\\&=
  \frac{1}{2^{\lambda}} \sum_{\varepsilon=0}^{1}
  \sum_{u < 2^{\lambda-1}} \e\p{\frac{1}{2}\sum_{\ell<K} b_\ell
	s_{\lambda}(2u + \varepsilon + \ell(2d + \delta) + i(\ell))
  - h(2u+\varepsilon)2^{-\lambda}}
\\&=
  \frac{1}{2^{\lambda}} \sum_{\varepsilon=0}^{1} \e\p{
  -\frac{h\varepsilon}{2^{\lambda} } }
  \sum_{u < 2^{\lambda-1}}
  \e\Bigg(\frac{1}{2}\sum_{\ell<K} b_\ell
	\bigg(
  s_{\lambda-1}\p{u + \ell d + \floor{\frac{i(\ell) + \ell \delta + \varepsilon}{2} } }
\\&&
\mathmakebox[0pt][r]{
  + \varepsilon_0(i(\ell) + \ell \delta + \varepsilon)\biggr)
  - hu2^{\lambda-1}\Biggr)
}
\\&=
  \frac 12 \sum_{\varepsilon=0}^{1}\e\p{-\frac{h\varepsilon}{2^{\lambda} } }f_{\delta,\varepsilon}^{i,b}\, G_{\lambda-1}^{T_{\delta,\varepsilon}(i),b}(h,d).
\end{align*}
\end{proof}
We want to study compositions of elementary transformations $T_{\delta,\varepsilon}$.
We therefore extend this notation as follows.
Assume that $d,e,m\geq 0$ are integers.
For $d,e<2^m$, we define $T_{d,e}^{(m)}:\mathcal I_K\ra\mathcal I_K$ by setting, for $i\in\mathcal I_K$ and $\ell<K$,
\begin{equation}\label{eqn:general_transformation_def}
T_{d,e}^{(m)}(i)(\ell)
=\floor{\frac{i(\ell)+\ell d+e}{2^m}}.
\end{equation}
For general integers $d,e\geq 0$ we set
\begin{equation*}
T_{d,e}^{(m)}
=T_{d\bmod 2^m,\,e\bmod 2^m}^{(m)}.
\end{equation*}
Note that we have
\begin{equation}\label{eqn:trafo_shift}
T_{d,0}^{(m)}(i)(\ell)\leq
T_{d,e}^{(m)}(i)(\ell)\leq
T_{d,0}^{(m)}(i)(\ell)+1
\end{equation}
for all $e\geq 0$.
By a straightforward induction it follows that
\begin{equation}\label{eqn:general_transformation_dec}
T_{d,e}^{(m)} = T_{\delta_{m-1},\varepsilon_{m-1}}\circ\cdots\circ T_{\delta_0,\varepsilon_0},
\end{equation}
for $m\geq 1$, where $\sum_{i\geq 0}\delta_i2^i$ and $\sum_{i\geq 0}\varepsilon_i2^i$
are the binary expansions of $d$ and $e$ respectively.
Moreover, $T^{(0)}_{d,e}$ is the identity on $\mathcal I_K$.
We also define, generalizing the notation
$f_{\delta,\varepsilon}^{i,b}$,
\begin{equation}\label{eqn:complex_weight_def}
f_{d,e}^{(m),i,b} = f_{\delta_{m-1},\varepsilon_{m-1}}^{T_{d,e}^{(m-1)}(i),b}\cdots f_{\delta_1,\varepsilon_1}^{T_{d,e}^{(1)}(i),b}\cdot f_{\delta_0,\varepsilon_0}^{i,b}.
\end{equation}
In order to obtain a recurrence relation for $\lvert G_{\lambda}^{i,b}(h,d)\rvert^2$, we define
\begin{equation*}
  \Phi_{\lambda}^{i_1,i_2,b}(h,d) = G_{\lambda}^{i_1,b}(h,d) \overline{G_{\lambda}^{i_2,b}(h,d)}.
\end{equation*}
Using \eqref{eqn:rec_G}, this immediately yields
\begin{multline}\label{eqn:rec_psi}
  \Phi_{\lambda}^{i_1,i_2,b}(h,2d+\delta)
\\
= \frac{1}{4} \sum_{\varepsilon_1<2}
  \sum_{\varepsilon_2<2} \e\p{-\frac{(\varepsilon_1-\varepsilon_2)h}{2^{\lambda} } }
  f_{\delta,\varepsilon_1}^{i_1,b} \overline{f_{\delta,\varepsilon_2}^{i_2,b}}
  \Phi_{\lambda-1}^{T_{\delta,\varepsilon_1}(i_1), T_{\delta,\varepsilon_2}(i_2),b}(h,d)
\end{multline}
for $\delta\in\{0,1\}$, and applying this identity iteratively one gets, for $m\in \dN$ and $d'<2^m$,
\begin{equation}\label{eqn:rec_psi_iterated}
\begin{aligned}
    \Phi_{\lambda}^{i_1,i_2,b}(h,2^md+d')
&=  \frac{1}{4^m} \sum_{e_1<2^m} \sum_{e_2<2^m}
    \e\p{-\frac{(e_1-e_2)h}{2^{\lambda} } }
\\& \quad\times f^{(m),i_1,b}_{d',e_1}\overline{f^{(m),i_2,b}_{d',e_2}}
    \Phi_{\lambda-m}^{T_{d',e_1}^{(m)}(i_1),T_{d',e_2}^{(m)}(i_2),b}(h,d).
\end{aligned}
\end{equation}
Obviously this implies for all $d'<2^m$
\begin{equation}\label{eqn:trivial_estimate}
\begin{aligned}
  \abs{G_{\lambda}^{i,b}(h,2^md+d')}^2
&=\abs{\Phi_{\lambda}^{i,i,b}(h,2^md+d')}
\\&\leq
\max_{e_1,e_2<2^m} \abs{\Phi_{\lambda-m}^{T_{d',e_1}^{(m)}(i),T_{d',e_2}^{(m)}(i),b}(h,d)}
\\&=
\max_{e<2^m} \abs{G_{\lambda-m}^{T_{d',e}^{(m)}(i),b}(h,d)}^2,
\end{aligned}
\end{equation}
an estimate that is also valid for $m=0$.
\subsection{An estimate for Fourier coefficients}
In this section, we are concerned with such sequences $b$
originating from a sequence $a=(a_0,\ldots,a_{L-1})\in\{0,1\}^L$, where $1\leq L\leq r$, via the assignment
\begin{equation}\label{eqn:b_structure}
b_\ell=\begin{cases}
a_\ell,&0\leq \ell<L,\\
-a_{\ell-r},&r\leq \ell<L+r-1,\\
0,&\text{otherwise.}
\end{cases}
\end{equation}
That is, the sequence $b$ consists of two blocks, identical modulo $2$.
From now on, we assume that $b$ is such a sequence and that $K=L+r$ (such that $b_\ell$, for $\ell<K$,
captures all nonzero values).
For brevity, and since $b$ is constant in what follows, we omit $b$ as an upper index of $G,\Phi$ and $f$.
Moreover, we assume throughout this section that $\lambda\geq 0,r\geq 1$ and $m\geq 5$ are integers such that
\begin{align}
&2^{m-5}\leq L<2^{m-4},\label{eqn:m_size}\\
&2m\leq\nu_2(r)\leq \lambda/4\label{eqn:r_divisibility}.
\end{align}
For brevity, we write $x=\nu_2(r)$.
\begin{lem}\label{lem:saving}
Assume that the sequence $i\in \mathcal{I}_K$ satisfies
\begin{equation}\label{eqn:prepared}
i(r)\bmod 2^m\in\{1,2\}.
\end{equation}
Let $z\geq 0$ and $h$ be integers and $0\leq d<2^z$. Then
\[
  \abs{G_{z+m}^i(h,d\,2^{m}+1)}^2
  \leq
  (1-\eta) \max_{e<2^m}
  \abs{G_z^{T^{(m)}_{1,e}(i)}(h,d)}^2
\]
for $\eta = \frac{2}{4^m}$.
\end{lem}
\begin{proof}
We rewrite the left hand side via the identity~\eqref{eqn:rec_psi_iterated}, setting $i_1=i_2=i$, and want to find pairs of indices $(e'_1,e'_2), (e''_1, e''_2)$ such that the corresponding two summands on the right hand side of~\eqref{eqn:rec_psi_iterated} cancel.
This will give the announced saving.
That is, we want
\begin{align}
T^{(m)}_{1,e'_1}(i) &= T^{(m)}_{1,e''_1}(i),
\label{eqn:saving_1}\\
T^{(m)}_{1,e'_2}(i) &= T^{(m)}_{1,e''_2}(i),
\label{eqn:saving_2}\\
f^{(m),i}_{1,e'_1}\overline {f^{(m),i}_{1,e'_2}} &= -f^{(m),i}_{1,e''_1}\overline{f^{(m),i}_{1, e''_2}}
\label{eqn:saving_3}\quad\text{and}\\
e'_1-e'_2&=e''_1-e''_2.\label{eqn:saving_4}
\end{align}
We show that these conditions are satisfied for the choice
\begin{align*}
    e'_1  = (0101^{m-3})_2, &\quad e'_2  = (1001^{m-3})_2,\\
    e''_1 = (0111^{m-3})_2, &\quad e''_2 = (1011^{m-3})_2,
\end{align*}
where $(\varepsilon_\nu\ldots\varepsilon_0)_2=\sum_{i\leq \nu}\varepsilon_i2^i$ and $1^k$ means $k$-fold repetition of the digit $1$.
Condition~\eqref{eqn:saving_4} is clearly true.
In order to verify~\eqref{eqn:saving_1} and~\eqref{eqn:saving_2}, we note that the binary representations of $e_1',e_2',e_1'',e_2''$ all start with $m-3$ ones.
We therefore define
\[    j=T_{1,(1^{m-3})_2}^{(m-3)}(i).    \]
By~\eqref{eqn:m_size}, which implies $i(\ell)+\ell<2^{m-3}$ for $0\leq\ell<L$, we have
\begin{equation}\label{eqn:j_first_part}
j(\ell)=\floor{\frac{i(\ell)+\ell+2^{m-3}-1}{2^{m-3} } }
=
\begin{cases}
0,&\ell=0,\\
1,&1\leq \ell<L.
\end{cases}
\end{equation}
Moreover,
by~\eqref{eqn:prepared}
we obtain $i(r+\ell)+\ell-1\bmod 2^m\in\{0,\ldots,2L-1\}$.
Using also~\eqref{eqn:m_size} and~\eqref{eqn:r_divisibility}, we get
\begin{equation}\label{eqn:j_second_part}
j(r+\ell)
=\floor{\frac{i(r+\ell)+r+\ell+2^{m-3}-1}{2^{m-3} } }
\equiv 1\bmod 8
\end{equation}
for $0\leq \ell<L$.
Since $j\in \mathcal I_K$, this equation implies that the value $j(\ell)$ is constant for 
$\ell\in[r,r+L)$.                                                                         
By~\eqref{eqn:general_transformation_dec} we obtain~\eqref{eqn:saving_1} and~\eqref{eqn:saving_2} as soon as we show that
$T_{0,2}^{(3)}(j)=T_{0,3}^{(3)}(j)=T_{0,4}^{(3)}(j)=T_{0,5}^{(3)}(j)$.
Using~\eqref{eqn:j_first_part} and~\eqref{eqn:j_second_part} we have for $0\leq e\leq 6$ and $0\leq \ell<L$
\begin{align*}
T_{0,e}^{(3)}(j)(\ell)&=\floor{\frac{j(\ell) + e}{8}}\leq\floor{\frac 78} = 0,\\
T_{0,e}^{(3)}(j)(r+\ell)&=\floor{\frac{j(r+\ell) + e}{8}} = \floor{\frac{j(r)-1}8+\frac {e+1}8}=\frac{j(r)-1}8.
\end{align*}
It remains to verify~\eqref{eqn:saving_3}, which is clearly equivalent to
\[
f^{(m),i}_{1,e'_1}\overline {f^{(m),i}_{1,e''_1}}=-f^{(m),i}_{1,e'_2}\overline{f^{(m),i}_{1,e''_2}},
\]
since the weights have absolute value $1$.
By~\eqref{eqn:general_transformation_dec},~\eqref{eqn:complex_weight_def} and the definition of $e_1',e_2',e_1'',e_2''$ this equation is equivalent to
\begin{equation}\label{eqn:weight_equality}
f^{(3),j}_{0,2}\overline {f^{(3),j}_{0,3}}=-f^{(3),j}_{0,4}\overline{f^{(3),j}_{0,5}}.
\end{equation}
By~\eqref{eqn:b_structure} we have, for any sequence $i\in \mathcal I_K$ and all $\delta,\varepsilon\in\{0,1\}$,
\begin{equation}\label{eqn:weight_formula}\nonumber
f_{\delta,\varepsilon}^{i}
=
\e\p{\frac 12\sum_{\ell<L}a_\ell(i(\ell)-i(r+\ell))}.
\end{equation}
Using this identity,~\eqref{eqn:j_first_part},~\eqref{eqn:j_second_part} and the assumption $a_0=1$ the verification of~\eqref{eqn:weight_equality} is straightforward, which completes the proof.
\end{proof}
Let $d=(d_{\lambda-1}\cdots d_0)_2<2^\lambda$.
We call a position $\mu$ {\it good} (a notion that depends on $\lambda,d,i,r$ and $m$), if the following properties are satisfied.
\begin{enumerate}[(a)]   
\item $0\leq \mu\leq \lambda-m$. \label{it:good_a}
\item $(d_{\mu+m-1},\ldots,d_{\mu+1},d_\mu)=(0,\ldots,0,1)$. \label{it:good_b}
\item $T_{d,0}^{(\mu)}(i)(r)\equiv 1\bmod 2^m$. \label{it:good_c}
\end{enumerate}
The point of this notion is that at each good position, using \lemref{lem:saving}, we win a factor $1-\eta$ in the estimate of $\Phi(h,d)$. This argument is carried out in the following lemma.
\begin{lem}\label{lem:single_estimate}
Let $i\in \mathcal I_K$ and $k\geq 0$ and assume that $d<2^\lambda$ is such that the number of good positions $\mu$ is at least $k$. Then
  \begin{align*}
    \abs{G_{\lambda}^{i}(h,d)}^2 \leq (1-\eta)^{k}
  \end{align*}
  holds for all $h\in\dZ$ and $\eta = 2/4^m$.
\end{lem}
\begin{proof}
Let $0\leq\mu_0<\ldots<\mu_{k-1}$ be good positions.
We set $\mu_k=\lambda$.
The estimate~\eqref{eqn:trivial_estimate} implies that
\begin{align*}
  \abs{G_{\lambda}^{i}(h,d)}^2 \leq \max_{e<2^{\mu_0}}
  \abs{G_{\lambda-\mu_0}^{T^{(\mu_0)}_{d,e}(i)}(h,\floor{d/2^{\mu_0} } ) }^2.
\end{align*}
Let $0\leq j<k$.
Since the position $\mu_j$ is good, we know by~\eqref{it:good_c} and~\eqref{eqn:trafo_shift} that $T^{(\mu_j)}_{d,e}(r)\bmod 2^m\in\{1,2\}$ for all $e<2^{\mu_j}$.
Thus we can apply \lemref{lem:saving}, the identity~\eqref{eqn:general_transformation_dec} and the estimate~\eqref{eqn:trivial_estimate} and obtain
\begin{align*}
\hspace{4em}&\hspace{-4em}\max_{e<2^{\mu_j}}
\abs{G_{\lambda-\mu_j}^{T^{(\mu_j)}_{d,e}(i)}(h,\floor{d/2^{\mu_j} } ) }^2
\\&\leq
  \max_{e<2^{\mu_j}} (1-\eta) \max_{e'<2^m}\abs{G_{\lambda-\mu_j-m}^{T^{(m)}_{1,e'} \circ T^{(\mu_j)}_{d,e}(i)}\p{h,\floor{d/2^{\mu_j+m} } } }^2
\\&= (1-\eta) \max_{e<2^{\mu_j+m}} \abs{G_{\lambda-\mu_j-m}^{T^{(\mu_j+m)}_{d,e}(i)}\p{h,\floor{d/2^{\mu_j+m} } } }^2
\\&\leq (1-\eta) \max_{e<2^{\mu_{j+1} } }
\abs{G_{\lambda-\mu_{j+1}}^{T^{(\mu_{j+1})}_{d,e}(i)}\p{h,\floor{d/2^{\mu_{j+1} } } } }^2.
\end{align*}
This proves the desired upper bound.
\end{proof}
In order to show that for most $d$ there are many good positions, we have a closer look at condition~\eqref{eqn:prepared}.
\begin{lem}\label{lem:prepare}
Write $r=2^xr_0$ with $r_0$ odd and assume that $y\geq 0$ and $0\leq d_0<2^y$.
Let $i\in \mathcal I_K$.
There exists a unique $d_1\in \{0,\ldots,2^m-1\}$ such that for all $d_2\in \{0,\ldots,2^{x-m}-1\}$ we have
\[  T_{d,0}^{(x+y)}(i)(r)\equiv 1\bmod 2^m,  \]
where $d=2^{y+m}d_2+2^yd_1+d_0$.

If $d_1'<2^y$ is different from $d_1$,
we have $T_{d,0}^{(x+y)}(i)(r)\not\equiv 1\bmod 2^m$ for all $d_2\in \{0,\ldots,2^{x-m}-1\}$.
\end{lem}
\begin{proof}
Since $r_0$ is odd, the statements follow from
\[
T_{d,0}^{(x+y)}(i)(r)
=
\floor{\frac{i(r)}{2^{x+y}}+\frac{r_0d_0}{2^y}}+r_0d_1+r_02^md_2
\equiv
\floor{\frac{i(r)}{2^{x+y}}+\frac{r_0d_0}{2^y}}+r_0d_1\bmod 2^m.
\]
\end{proof}
Note that the good-ness of a position $\mu$ does not depend on the digits of $d$ with indices $\mu-x+m,\ldots,\mu-1$.
Let $\lambda\geq 0$.
We decompose the set $\{0,\ldots,\lambda-1\}$ into intervals as follows.
Consider the mutually disjoint sets of indices
\begin{equation*}
\begin{aligned}
A_1&=\{2\ell_1 x+\ell_0m: 0\leq\ell_1<\floor{\lambda/(2x)}\text{ and }0\leq\ell_0<\floor{x/m}\},\\
A_2&=\{(2\ell_1+1) x+\ell_0m: 0\leq\ell_1<\floor{\lambda/(2x)}\text{ and }0\leq\ell_0<\floor{x/m}\},
\end{aligned}
\end{equation*}
which form the starting points of intervals of length $m$. We call these intervals to be {\it of type~1} and~{\it 2} respectively.
The integers in 
$[0,\lambda)$ 
not contained in an interval of type~1 or~2 form intervals {\it of type~3}, having total length $\lambda-2m\floor{\lambda/(2x)}\floor{x/m}$.
Assume that $\lambda\geq 2x$, which will be guaranteed by the hypotheses of \propref{prp:uniform_fourier}.
Then, beginning at~$0$, the resulting partition starts with $\floor{x/m}$ intervals of type~1, followed possibly (if and only if $m\nmid x$) by an interval of type~3 , which fills up the gap up to position~$x$.
This is followed by~$\floor{x/m}$ intervals of type~2 and possibly an interval of type~3, reaching~$2x$. This pattern continues up to~$2x\floor{\lambda/(2x)}$, the last interval of type~3 however extends up to $\lambda$.
\begin{lem}\label{lem:number_good}
Let $M$ be a $k$-element subset of $A_2$.
The number of $d<2^\lambda$ such that $M$ is the set of good positions in $A_2$
equals $2^{\lambda-2m\lambda_0}(2^{2m}-1)^{\lambda_0-k}$, where $\lambda_0=\lvert A_1\rvert=\lvert A_2\rvert=\floor{\lambda/(2x)}\floor{x/m}$.
\end{lem}
\begin{proof}
We construct recursively the set of admissible $d=(d_{\lambda-1}\cdots d_0)_2<2^\lambda$.
In order to do so, we let $\mu$ run through the set $A_1\cup A_2\cup A_3$ in ascending order and choose digits of $d$ in such a way that all digits up to position $\mu$ have already been chosen when we reach $\mu$.
If we encounter an index $\mu\in A_1$, we set the $2m$ digits $d_{\mu},\ldots,d_{\mu+m-1},d_{\mu+x},\ldots,d_{\mu+x+m-1}$ according to two cases:
if $\mu+x\in M$, we have to guarantee good-ness of the position $\mu+x$,
we therefore set $(d_{\mu+m-1},\ldots,d_{\mu})$ according to \lemref{lem:prepare} and $(d_{\mu+x+m-1},\ldots,d_{\mu+x})=(0,\ldots,0,1)$.
If $\mu+x\not\in M$, we may choose any of the remaining $2^{2m}-1$ possibilities for these $2m$ digits, such that the position $\mu+x$ will not be good.
If we encounter $\mu\in A_2$, we do nothing, since the corresponding block of length~$m$ has already been filled.
If we find an index $\mu\in A_3$, it is the starting point of an interval $I$ of type~3.
We may set the digits of $d$ with indices in $I$ freely.
Therefore we can choose $\lambda-2m\lvert A_1\rvert$ digits arbitrarily, and $\lvert A_1\rvert-k$ times we may choose out of $2^{2m}-1$ possibilities. This implies the statement of the lemma.
\end{proof}
The proof of \propref{prp:uniform_fourier} is now easy.
\begin{proof}[Proof of Proposition \ref{prp:uniform_fourier}]
Let $\lambda_0=\floor{\lambda/(2x)}\floor{x/m}$.
It follows from \lemref{lem:number_good} that there are precisely $\binom{\lambda_0}{k} 2^{\lambda-2m\lambda_0} \bigl(2^{2m}-1\bigr)^{\lambda_0-k}$ integers $d\in\{0,\ldots,2^\lambda-1\}$ such that exactly~$k$ positions from $A_2$ are good.
Since each $d$ occurs for some $k\leq \lambda_0$, \lemref{lem:single_estimate} implies
\begin{align*}
  \sum_{d < 2^{\lambda}} \max_{h<2^{\lambda}} \abs{G_{\lambda}^{i}(h,d)}^2
&\leq
  2^{\lambda-2m\lambda_0}\sum_{k=0}^{\lambda_0} \binom{\lambda_0}{k}(2^{2m}-1)^{\lambda_0-k} (1-\eta)^{k}
\\&=
  2^{\lambda} (1-2/16^m)^{\lambda_0}.
\end{align*}

By~\eqref{eqn:r_divisibility} we obtain
\[
\lambda_0=\floor{\lambda/(2x)}\floor{x/m}\geq
\frac{\lambda-2x}{2x}\frac{x-m}{m} \geq \frac{\lambda/2}{2x}\frac{x/2}{m} = \frac{\lambda}{8m}.
\]
This finishes the proof.
\end{proof}
\section{Proof of \propref{prp:PS_via_beatty_2}}\label{sec:proof_PS_via_beatty}
The following elementary lemma summarizes (and extends)
Lemmas~9 through~11 from~\cite{S2014}.
\begin{lem}\label{lem:approx_lemma}
Let $a,b$ be real numbers such that $a\leq b$ and set $K=b-a$.
Assume that $f:[a,b]\ra\dR$ is twice differentiable and $\abs{f''}\leq B$.
Then for all $\alpha\in f'([a,b])$ the following statements hold.
\begin{enumerate}[(i)]
\item For $a\leq x\leq b$ we have
\[    \abs{x\alpha+f(a)-a \alpha-f(x)} \leq BK^2.    \]
\item If $x\in [a,b]$ is such that
$\norm{x\alpha+f(a)-a\alpha}>BK^2$, then
\[     \floor{f(x)}=\floor{x\alpha+f(a)-a\alpha}.    \]
\item If $a,b$ are integers, $L\geq 1$ is an integer,
 $f:[a,b+L-1]\ra\dR$ is twice differentiable, $\alpha\in f'([a,b])$ and $\lvert f''\rvert\leq B$,
then
\begin{multline*}
\abs{\{n\in 
(a,b]:      
\floor{f(n+\ell)}\neq\floor{(n+\ell)\alpha+f(a)-a\alpha}\text{ for some }\ell<L\}}
\\\leq 2B(K+L-1)^3L+KLD_K(\alpha).
\end{multline*}
\end{enumerate}
\end{lem}
We prove \propref{prp:PS_via_beatty_2}.
For convenience, let $\varphi(n)=0$ for $n\leq 0$ and set $f(n)=n^c$.
We follow the proof of~\cite[Proposition~1]{S2014}.
By analogous considerations as given there, we may assume that $K$ is an integer and that $2\leq K\leq N$.

Define integral partition points $a_i=\ceil{N}+iK$ for $i\geq 0$
and set $M=\max\{i:a_i+L-1\leq 2N\}$.
The integer $M$ satisfies the estimate $KM\leq N$.
We have the decomposition
\begin{equation}\label{eqn:interval_decomposition}
(N,2N]=\left(N,\ceil{N}\right]\cup\bigcup_{0\leq i<M}(a_i,a_{i+1}]\cup(a_M,2N].
\end{equation}
Let $\alpha\in\dR$.
Then by the triangle inequality and the relation $a_{i+1}-a_i=K$ we have for $i<M$
\begin{multline}\label{eqn:elementary_split}
    \bigl\lvert
        \abs{
        \{                    
          n\in (a_i,a_{i+1}]: 
          \varphi\p{\floor{(n+\ell)^c}}=\omega_\ell
          \text{ for }0\leq \ell<L
        \}
        }
      -
        K\delta
    \bigr\rvert
\\
  \leq
      T_1(\alpha,i)
    +
      T_2(\alpha,i)
,
\end{multline}
where
\begin{align*}
  T_1(\alpha,i)
&=
    \bigl\lvert
        \abs{
        \{
          n\in (a_i,a_{i+1}]:
          \varphi\p{\floor{(n+\ell)^c}}=\omega_\ell
          \text{ for }0\leq \ell<L
        \}
        }
\\&
    -\abs{\{
      n\in (a_i,a_{i+1}]:
      \varphi\p{\floor{(n+\ell)\alpha+f(a_i)-a_i\alpha}}=\omega_\ell
      \text{ for }0\leq \ell<L
    \}}\bigr\rvert
\\
    T_2(\alpha,i)
  &=
    \bigl\lvert
      \abs{
        \{
          n\in (a_i,a_{i+1}]:
          \varphi\p{\floor{(n+\ell)\alpha+f(a_i)-a_i\alpha}}=\omega_\ell
          \text{ for }0\leq \ell<L
        \}
      }
\\&
      -
      K\delta
    \bigr\rvert
\end{align*}
We integrate both sides of
~\eqref{eqn:elementary_split} in the variable
$\alpha$ from $f'(a_i)$ to $f'(a_{i+1})$,
divide by the length of the integration range,
and take the sum over $i$ from $0$ to $M-1$, which yields
\begin{multline}\label{eqn:split}
    \bigl\lvert
        \abs{
        \{                             
          n\in \left(a_0,a_M\right]:   
          \varphi\p{\floor{(n+\ell)^c}}=\omega_\ell
          \text{ for }0\leq \ell<L
        \}
        }
      -
        MK\delta
    \bigr\rvert
\\
  \leq
    \sum_{0\leq i<M}
      \frac 1{f'(a_{i+1})-f'(a_i)}
      \int_{f'(a_i)}^{f'(a_{i+1})}
        \left(
          \vphantom{\sum}
          T_1(\alpha,i)
          +T_2(\alpha,i)
        \right)
      \,\mathrm d\alpha
.
\end{multline}
We estimate the first summand. 
If $0\leq i<M$ and $\alpha\in f'([a_i,a_{i+1}])$, \lemref{lem:approx_lemma} gives
\begin{equation}\label{eqn:beatty_erdos_turan}
    T_1(\alpha,i)  \leq  2f''(N)(K+L-1)^3L + LKD_K(\alpha).
\end{equation}
By the Mean Value Theorem we have
\begin{align}\label{eqn:mvt_stretch}
\frac 1{f'(a_{i+1})-f'(a_i)}\ll \frac NK\frac 1{f'(2N)-f'(N)}
\end{align}
for $0\leq i<M$.
Using this and the integral mean discrepancy estimate~\eqref{eqn:mean_discrepancy_int} we obtain
\begin{align*}\label{eqn:kuzmin_landau_integral}
\hspace{4em}&\hspace{-4em}
    \sum_{0\leq i<M}{
      \frac 1{f'(a_{i+1})-f'(a_i)}
      \int_{f'(a_i)}^{f'(a_{i+1})}
        D_K(\alpha)
      \,\mathrm d\alpha
    }
\\
  &\leq
    \frac NK\frac 1{f'(2N)-f'(N)}
    \sum_{0\leq i<M}{
      \int_{f'(a_i)}^{f'(a_{i+1})}
      D_K(\alpha)
      \,\mathrm d\alpha
    }
\\
  &\leq
    \frac NK
    \frac{f'(2N)-f'(N)+1}{f'(2N)-f'(N)}
      \int_0^1
      D_K(\alpha)
      \,\mathrm d\alpha
\\
  &\ll
    \Bigl(N+\frac 1{f''(N)}\Bigr)\frac{\bigl(\logp K\bigr)^2}{K^2}
.
\end{align*}
By the estimates $KM\leq N$
and $N\geq C/f''(N)$
this implies
\begin{multline*}
    \sum_{0\leq i<M}{
      \frac 1{f'(a_{i+1})-f'(a_i)}
      \int_{f'(a_i)}^{f'(a_{i+1})}{
        T_1(\alpha,i)
      }
      \,\mathrm d \alpha
    }
  \\\leq
    C_1NL\biggl(
        f''(N)K^2
      +
        \frac{\bigl(\logp N\bigr)^2}{K}
    \biggr)
\end{multline*}
for some constant $C_1$ depending on $c$ and $L$.
We turn our attention to the second summand in~\eqref{eqn:split}.
Inserting~\eqref{eqn:mvt_stretch} and the definition of $T_2(\alpha,i)$, we easily obtain
\begin{equation}\label{eqn:second_step}
    \sum_{0\leq i<M}{
      \frac{1}{f'(a_{i+1})-f'(a_i)}
      \int_{f'(a_i)}^{f'(a_{i+1})}{
        T_2(\alpha,i)
      }
      \,\mathrm d\alpha
    }
  \ll
    N\,J(N,K)
.
\end{equation}
Estimating also the contributions of the first and the last interval in~\eqref{eqn:interval_decomposition} trivially and collecting the error terms,
we obtain
\begin{multline*}
    \abs{
        \frac 1N
        \bigl\lvert
          \bigl\{                
              n\in (N,2N]:       
              \varphi\bigl(\bigl\lfloor(n+\ell)^c\bigr\rfloor\bigr)=\omega_\ell
              \text{ for }0\leq \ell<L
          \bigr\}
        \bigr\rvert
      -
        \delta
    }
\\
  \leq
    C_2
    \biggl(
        f''(N)K^2
      +
        \frac{(\log N)^2}{K}
      +
        J(N,K)
      +
        \frac KN
    \biggr)
,
\end{multline*}
for some $C_2$ depending on $c$ and $L$.
By the estimate $f''(N)\geq C/N$ the last term is dominated by the first, which finishes the proof of \propref{prp:PS_via_beatty_2}.
\bibliographystyle{siam}
\bibliography{normality}

\def\cprime{$'$}
\begin{thebibliography}{10}

\bibitem{AS99}
{\sc J.-P. Allouche and J.~Shallit}, {\em The ubiquitous
  {P}rouhet-{T}hue-{M}orse sequence}, in Sequences and their applications
  ({S}ingapore, 1998), Springer Ser. Discrete Math. Theor. Comput. Sci.,
  Springer, London, 1999, pp.~1--16.

\bibitem{AS2003}
\leavevmode\vrule height 2pt depth -1.6pt width 23pt, {\em Automatic
  sequences}, Cambridge University Press, Cambridge, 2003.
\newblock Theory, applications, generalizations.

\bibitem{B65}
{\sc E.~Bombieri}, {\em On the large sieve}, Mathematika, 12 (1965),
  pp.~201--225.

\bibitem{B87}
{\sc S.~Brlek}, {\em Enumeration of factors in the {T}hue-{M}orse word},
  {Discrete Appl. Math.}, 24 (1989), pp.~83--96.

\bibitem{DT2006}
{\sc C.~Dartyge and G.~Tenenbaum}, {\em Congruences de sommes de chiffres de
  valeurs polynomiales}, Bull. London Math. Soc., 38 (2006), pp.~61--69.

\bibitem{DDM2012}
{\sc J.-M. Deshouillers, M.~Drmota, and J.~F. Morgenbesser}, {\em Subsequences
  of automatic sequences indexed by {$\lfloor n^c\rfloor$} and correlations},
  J. Number Theory, 132 (2012), pp.~1837--1866.

\bibitem{DMR2015}
{\sc M.~Drmota, C.~Mauduit, and J.~Rivat}, {\em The {T}hue-{M}orse sequence
  along squares is normal}.
\newblock Submitted.

\bibitem{DM2012}
{\sc M.~Drmota and J.~F. Morgenbesser}, {\em Generalized {T}hue-{M}orse
  sequences of squares}, Israel J. Math., 190 (2012), pp.~157--193.

\bibitem{EH70}
{\sc P.~D. T.~A. Elliott and H.~Halberstam}, {\em A conjecture in prime number
  theory}, in Symposia {M}athematica, {V}ol. {IV} ({INDAM}, {R}ome, 1968/69),
  Academic Press, London, 1970, pp.~59--72.

\bibitem{FM96}
{\sc E.~Fouvry and C.~Mauduit}, {\em Sommes des chiffres et nombres presque
  premiers}, Math. Ann., 305 (1996), pp.~571--599.

\bibitem{G68}
{\sc A.~O. Gel{\cprime}fond}, {\em Sur les nombres qui ont des propri\'et\'es
  additives et multiplicatives donn\'ees}, Acta Arith., 13 (1967/1968),
  pp.~259--265.

\bibitem{HW54}
{\sc G.~H. Hardy and E.~M. Wright}, {\em An introduction to the theory of
  numbers}, Oxford, at the Clarendon Press, 1954.
\newblock 3rd ed.

\bibitem{M2001}
{\sc C.~Mauduit}, {\em Multiplicative properties of the {T}hue-{M}orse
  sequence}, Period. Math. Hungar., 43 (2001), pp.~137--153.

\bibitem{MR95}
{\sc C.~Mauduit and J.~Rivat}, {\em R\'epartition des fonctions
  {$q$}-multiplicatives dans la suite {$([n^c])_{n\in\bold N},\ c>1$}}, Acta
  Arith., 71 (1995), pp.~171--179.

\bibitem{MR2005}
\leavevmode\vrule height 2pt depth -1.6pt width 23pt, {\em Propri\'et\'es
  {$q$}-multiplicatives de la suite {$\lfloor n^c\rfloor$}, {$c>1$}}, Acta
  Arith., 118 (2005), pp.~187--203.

\bibitem{MR2009}
\leavevmode\vrule height 2pt depth -1.6pt width 23pt, {\em La somme des
  chiffres des carr\'es}, Acta Math., 203 (2009), pp.~107--148.

\bibitem{MR2010}
\leavevmode\vrule height 2pt depth -1.6pt width 23pt, {\em Sur un probl\`eme de
  {G}elfond: la somme des chiffres des nombres premiers}, Ann. of Math. (2),
  171 (2010), pp.~1591--1646.

\bibitem{M78}
{\sc H.~L. Montgomery}, {\em The analytic principle of the large sieve}, Bull.
  Amer. Math. Soc., 84 (1978), pp.~547--567.

\bibitem{M2007}
{\sc Y.~Moshe}, {\em On the subword complexity of {T}hue-{M}orse polynomial
  extractions}, Theoret. Comput. Sci., 389 (2007), pp.~318--329.

\bibitem{P53}
{\sc I.~I. {Piatetski-Shapiro}}, {\em On the distribution of prime numbers in
  sequences of the form {$[f(n)]$}}, Mat. Sbornik N.S., 33(75) (1953),
  pp.~559--566.

\bibitem{RS2001}
{\sc J.~Rivat and P.~Sargos}, {\em Nombres premiers de la forme {$\lfloor
  n^c\rfloor$}}, Canad. J. Math., 53 (2001), pp.~414--433.

\bibitem{S2014}
{\sc L.~Spiegelhofer}, {\em Piatetski-{S}hapiro sequences via {B}eatty
  sequences}, Acta Arith., 166 (2014), pp.~201--229.

\bibitem{S2015}
\leavevmode\vrule height 2pt depth -1.6pt width 23pt, {\em Normality of the
  {T}hue-{M}orse sequence along {P}iatetski-{S}hapiro sequences}, Q. J. Math.,
  66 (2015), pp.~1127--1138.

\end{thebibliography}
\end{document}